\documentclass[pdflatex,sn-mathphys-ay]{sn-jnl}

\usepackage{graphicx}%
\usepackage{multirow}%
\usepackage{amsmath,amssymb,amsfonts}%
\usepackage{amsthm}%
\usepackage{mathrsfs}%
\usepackage[title]{appendix}%
\usepackage{xcolor}%
\usepackage{textcomp}%
\usepackage{manyfoot}%
\usepackage{booktabs}%
\usepackage{algorithm}%
\usepackage{algorithmicx}%
\usepackage{algpseudocode}%
\usepackage{listings}%

\usepackage{algorithm}
\usepackage{algpseudocode}

\theoremstyle{thmstyleone}%
\newtheorem{theorem}{Theorem}
\newtheorem{proposition}{Proposition}%

\theoremstyle{thmstyletwo}%
\newtheorem*{remarks}{Remarks}%

\theoremstyle{thmstylethree}%
\newtheorem{definition}{Definition}%

\newcommand{\RR}{\mathbb{R}}

\newcommand{\cP}{\mathcal{P}}
\renewcommand{\figurename}{Fig.}

\DeclareMathOperator{\sym}{Sym}

\renewcommand{\hat}{\widehat}
\renewcommand{\tilde}{\widetilde}
\renewcommand{\Re}{\operatorname{Re}}
\renewcommand{\Im}{\operatorname{Im}}
\DeclareMathOperator*{\argmin}{argmin}

\newenvironment{lemma*}[1][]
{\par\medskip\noindent\textbf{Lemma\if\relax\detokenize{#1}\relax\else\ (#1)\fi.}\ \itshape}
{\par\medskip}

\usepackage{multirow,bbm,bm}
\usepackage{subfig}
\usepackage{float}

\usepackage{placeins}

\newenvironment{appendixcompact}{%
  \begingroup
  \setlength{\parskip}{0pt}
  \setlength{\parindent}{1em}
  \setlength{\abovedisplayskip}{5pt plus 2pt minus 2pt}
  \setlength{\belowdisplayskip}{5pt plus 2pt minus 2pt}
  \setlength{\textfloatsep}{8pt plus 2pt minus 2pt}
  \setlength{\intextsep}{6pt plus 2pt minus 2pt}
  \setlength{\floatsep}{6pt plus 2pt minus 2pt}
}{%
  \endgroup
}

\begin{document}

\title[Multivariate mixture distributions]{Two statistical problems for multivariate mixture distributions}

\author[1]{\fnm{Fraiman} \sur{Ricardo}}\email{fraimanricardo@gmail.com}

\author[2]{\fnm{Moreno} \sur{Leonardo}}\email{leonardo.moreno@fcea.edu.uy}
\equalcont{These authors contributed equally to this work.}

\author*[3]{\fnm{Ransford} \sur{Thomas}}\email{ransford@mat.ulaval.ca}
\equalcont{These authors contributed equally to this work.}

\affil[1]{\orgdiv{Departamento de Matem\'atica y Ciencias, \orgname{Universidad de San Andr\'es}, \country{Argentina}} and \orgdiv{Pedeciba Matemática, \orgname{} \country{Uruguay}}}

\affil[2]{\orgdiv{Instituto de Estad\'{i}stica, Departamento de M\'etodos Cuantitativos, \orgname{FCEA, Universidad de la Rep\'ublica}, \country{Uruguay}}}

\affil*[3]{\orgdiv{D\'epartement de math\'ematiques et de statistique}, \orgname{Universit\'e Laval}, \orgaddress{ \city{Qu\'ebec City}, \postcode{G1V 0A6}, \state{Qu\'ebec}, \country{Canada}}}

\abstract{
We address two important statistical problems: that of estimating mixtures of
multivariate normal distributions and mixtures of $t$-distributions based on
univariate projections, and that of quantifying a discrepancy between mixture
distributions induced by two model-based clusterings. In the second problem, rather than introducing a direct metric on partitions, we propose a model-based distributional discrepancy between the fitted mixture distributions associated with two clusterings.
The results are based on an earlier work of the authors, where it was shown that mixtures of multivariate 
Gaussian or $t$-distributions can be distinguished by projecting them onto a certain predetermined finite set of lines, the number of lines depending only on the total number of distributions  involved and on the ambient dimension.
We also compare our proposal with robust versions of the expectation-maximization method EM. In each case, we present algorithms for effecting the task, and compare them with existing methods  by carrying out some simulations.}




\maketitle

\section{Introduction}

\subsection{Random projections in statistics}

Random-projection (RP) techniques have become a cornerstone in high-dimensional statistical analysis. 
Over the last two decades, the theoretical and applied developments of RP have flourished. Early applications emphasized computational efficiency and scalability, particularly for clustering and matrix approximations \citep{BM2001,HMT2011}. The work presented in \citep{A2001} introduced more efficient sparse and structured random matrices, which proved crucial for large-scale implementations.

The best-known method for applications involving random projections is the one based on the Johnson--Lindenstrauss lemma, which has numerous applications, especially in image processing.
	Although there are many other proposals based on RP, in what follows we will briefly describe the Johnson--Lindenstrauss approach below and then focus solely on methods based on extensions of the Cram\'er--Wold theorem. We will describe each of them in more detail below.

\subsubsection{RP via the Johnson--Lindenstrauss lemma}

This lemma  can be stated as follows.

\begin{lemma*}[Johnson--Lindenstrauss lemma, \citep{JL1984}]
Let $\aleph_n$ be a set of $n$ points in $\mathbb{R}^d$, 
let $0 < \epsilon < 1$ and let $q>(8\log n)/\epsilon^2$. Then there exists a linear map
$f: \mathbb{R}^d \to \mathbb{R}^q$ such that, for all $u, v \in \aleph_n$,
\begin{equation}\label{E:JL}
(1-\epsilon)\|u-v \|^2 \le \| f(u) - f(v)\|^2 \le (1+\epsilon)\| u-v \|^2 .
\end{equation}
\end{lemma*}

Thus any set of points in high-dimensional Euclidean space can be linearly embedded into a much lower-dim\-ensional space such that pairwise distances are approximately preserved. 
This result lays the theoretical foundation for the widespread application of RP in modern data science and machine learning.

A natural question that arises is how to find the linear map $f$.
Dasgupta and Gupta \citep{DG2003}
propose a solution based on the following result,
which allows us to find $f$ in randomized polynomial time.

\begin{theorem}[Dasgupta--Gupta \citep{DG2003}]
Let $\aleph_n$ and $\epsilon$ be as  in Theorem~\ref{E:JL},
and let $q>(4\log n)/(\epsilon^2/2-\epsilon^3/3)$.
Define $f(u):= \sqrt{(d/q)}PUu$, where 
$U$ is a $d\times d$ unitary matrix
whose columns are iid (independent and identically distributed) random vectors 
uniformly distributed on the unit sphere in $\RR^d$,
and $P:\RR^d\to\RR^q$ denotes the projection onto the first $q$
coordinates.
Then
\[
\mathbb{P}\Bigl((1-\epsilon)\Vert u-v \Vert^2 \leq \Vert f(u) - f(v)\Vert^2 \leq (1+\epsilon)\Vert u-v \Vert^2)~\forall u,v\in\aleph_n\Bigr) \ge 1/n.
\]	
\end{theorem}

Recent studies continue to build on the Johnson--Lindenstrauss lemma, investigating its application in modern contexts, for example, as a method for dimensionality reduction  \citep{D2013}, particularly in machine-learning applications, in high-dimensional clustering, and nearest neighbor search. These studies demonstrate how RP can preserve important properties in high-dimensional spaces while dramatically reducing the computational complexity of various algorithms.

A nice property of the inequality \eqref{E:JL}
is that it does not depend on the dimension $d$. 
However,  if $\epsilon$ is small then $q$ becomes very large.

\subsubsection{RP via the Cram\'er--Wold theorem}

A more recent and statistically grounded direction emerged in the work of  Cuesta, Fraiman and Ransford \citep{CFR07},
based on extensions of the Cram\'er--Wold theorem.

In its original form, established in \citep{CW36}, the Cram\'er--Wold theorem says that
a multivariate distribution is uniquely determined by its one-dimensional projections. Subsequent work of R\'enyi \citep{Re52},
Gilbert \citep{Gi55}, B\'elisle--Mass\'e--Ransford \citep{BMR97} and Cuesta--Fraiman--Ransford \citep{CFR07}
showed that, under appropriate assumptions on the growth of its moments, a multivariate distribution is determined by 
a `sufficiently rich' set of lower-dimensional projections,
thereby allowing it to be determined from random projections onto low-dimensional subspaces.

Here is one such result. In what follows,
given a Borel probability measure $P$ on $\mathbb{R}^d$ and a vector subspace $H$ of $\mathbb{R}^d$, we denote  
by $P_H$ the projection of $P$ onto $H$, defined as the Borel probability measure on $H$ given by
\[
P_H(B):=P(\pi_H^{-1}(B)),
\]
where $\pi_H:\RR^d\to H$ is the orthogonal projection of $\RR^d$ onto $H$. Also, if $x\in\RR^d$, then $\langle x\rangle$ denotes the
vector subspace spanned by $x$.

\begin{theorem}[\protect{Cuesta--Fraiman--Ransford
\cite[Corollary 3.2]{CFR07}}]
Let $P$ and $Q$ be Borel probability measures on $\mathbb{R}^d$, where $d \geq 2$. Assume that:
\par
(i) the absolute moments $m_n = \int \|x\|^n\,dP(x)$ are finite and satisfy the  condition $\sum_{n=1}^{\infty} m_n^{-1/n}= \infty$;
\par
(ii)  the set $ \{x \in \mathbb{R}^d: P_{\langle x\rangle}= Q_{\langle x\rangle}\}$ is of positive Lebesgue measure.
\par
Then $P=Q.$
 \end{theorem}

Cuesta--Fraiman--Ransford focused on using RP not only as a computational tool but as a vehicle for inferring robust statistical structure. In particular, they demonstrated how statistical procedures \textit{random-projection tests} can preserve desirable properties such as consistency and robustness while remaining computationally feasible.

This line of work differs fundamentally from algorithmic RP approaches. Rather than using RP solely for dimensionality reduction, Cuesta, Fraiman  and their collaborators exploit the probabilistic distribution of projections to characterize centrality, perform goodness-of-fit tests \citep{CBFM2007}, depth measure \citep{CFF2007},  detect outliers  \citep{NC2021}, and hypothesis tests \citep{FMV2017}. For example, in their framework, distributions are projected onto randomly chosen lines, and classical univariate techniques are applied. Repeating this process over multiple projections yields statistical conclusions that are consistent with high-dimensional structures.

These methods have inspired extensions in several applied domains. For instance, random-projection depth was later adapted to functional data analysis, see \citep{CF2009} and \citep{CF2010}. 
The article \citep{NCG2014} proposes tests for functional data using projections, while \citep{BBTW2011} explores robust estimators via projection-pursuit methods. These studies reinforce the versatility of random projections as not only an algorithmic trick but a robust statistical principle. Moreover, this framework aligns conceptually with the work of \citep{DT2009}, which emphasized the geometric underpinnings of high-dimensional phenomena, although the focus
in \citep{DT2009} is often more combinatorial.

In contrast to mainstream applications of random projection, such as kernel approximation  or compressed sensing, the approach developed in \citep{CFR06,CFR07} is distinguished by its nonparametric nature and inferential focus. Rather than focusing on data transformations that preserve structural relationships for machine learning, the latter work is more oriented to preserving the geometric properties of the data for inferential purposes. This distinction in focus underscores the role of RP in nonparametric inference, as it provides an intuitive tool for deriving statistical conclusions from data without imposing stringent parametric assumptions.

More recently, there have been some further refinements 
of the Cram\'er--Wold theorem showing that, in some cases, just finitely many one-dimensional projections suffice to determine the measure. For example, the following result shows that,
in the case of elliptical distributions (which includes Gaussian and $t$-distributions), just $(d^2+d)/2$ projections suffice.

\begin{theorem}[\protect{Fraiman--Moreno--Ransford \citep{FMR23a}}]
Let $v_1,\dots,v_d$ be linearly independent vectors in $\mathbb{R}^d$, and let $S=\{v_j+v_k:1\le j\le k\le d\}$. If $P,Q$ are elliptical
distributions in $\mathbb{R}^d$ and if $P_{\langle x\rangle}=Q_{\langle x\rangle}$ for all $x\in S$, then $P=Q$.
\end{theorem}

In \citep{FMR2025} these results were extended to the case of certain multivariate mixtures. It was shown that mixtures of multivariate Gaussian or $t$-distributions can be distinguished by projecting them onto a certain predetermined finite set of lines, the number of lines depending only on the total number of distributions involved and on the ambient dimension. 
This last result is described in more detail in \S\ref{S:mixtures}, and its applications  form the subject of the rest of the paper.

\subsubsection{On the pros and cons of RP-methods}

Despite its many advantages, the RP methodology is not without its limitations. One notable drawback is that while it offers dimensionality reduction, it does so at the expense of certain properties of the data. Although the Johnson--Lindenstrauss lemma guarantees approximate distance preservation, it does not ensure exact preservation, especially when data points lie on complex, nonlinear manifolds. This means that RP may distort some aspects of the data distribution, especially for applications like clustering or classification, where fine-grained distinctions between data points are critical, see \citep{BM2001}.

Moreover, while RP has been shown to be robust in certain high-dimensional settings, its performance can be sensitive to the type of distribution or structure inherent in the data. For example, in datasets with heavy tails or outliers, RP techniques may not preserve the underlying distribution as effectively as more robust methods like those based on depth functions, see \citep{CFF2007}. Thus, while RP is often computationally efficient, it may sometimes fail to capture subtle patterns that are crucial in some statistical problems.

A significant challenge also arises from the inherent randomness of RP. This randomness introduces variability in the performance of algorithms across different projections. While multiple projections typically improve the consistency of results, determining the number of projections necessary for a reliable outcome is still an open question in many applications, see \citep{CFR07}. This issue has spurred research on the optimal number of projections required for specific tasks, particularly in the context of robust statistics and outlier detection.

Despite these challenges, the advantages of RP remain substantial. Its ability to reduce dimensionality while retaining much of the essential structure of the data makes it a powerful tool in large-scale statistical learning. Moreover, the simplicity of random projections allows for faster computations compared to other dimensionality reduction techniques, such as principal component analysis (PCA) or t-SNE, which can be computationally expensive in high-dimensional spaces, see \citep{VMH2008}.

As modern datasets continue to grow in complexity and dimensionality, the statistical insight provided by proj\-ection-based depth and testing becomes increasingly relevant. It offers a complement---and at times a counterpoint---to the more engineering-oriented uses of RP found in neural networks, sketching algorithms, or NLP embeddings. The blend of theoretical rigor and practical robustness in the works of Cuesta, Fraiman, Ransford and their collaborators provides tools that are especially suitable for applied statisticians working under real-world constraints, such as noisy data or nonstandard distributions.

Although the manuscript addresses: (i) parameter estimation for multivariate
mixtures and (ii) a distributional discrepancy between mixture distributions
induced by model-based clusterings, both contributions are governed by the same
principle: within the Gaussian/$t$-mixture classes considered here, the
underlying distribution is determined by a finite collection of one-dimensional
projections along a strong sm-uniqueness set (Section~2). Section~3 exploits
this fact constructively for estimation, whereas Section~4 exploits it
comparatively to quantify discrepancy by aggregating projection-wise
distributional differences.


 \subsection{Roadmap of the rest of the article}

In this article we consider two problems: that of estimating mixtures of
multivariate normal distributions and mixtures of $t$-distributions based on
univariate projections, and that of defining a discrepancy measure between
mixture distributions induced by two model-based clusterings. To solve these
problems we make use of some of the results given in \cite{FMR2025}.

Section~\ref{S:mixtures} begins with a short introduction to the notion of the Cram\'er--Wold device, and the main notions and results to be used are described.  

In Section~\ref{S:GMestimator}, we present a new procedure for  estimating  mixtures of multivariate normal distributions, and mixtures of $t$-distributions based on one-dimensional projections of the mixture based on characteristic functions (ECF), see \citep{Tr98} and \citep{XK10}, and using those univariate mixtures to provide consistent estimators of the parameters of the multivariate mixture.

The proposed random-projection method (RP) is compared with the well-known expectation-maxi\-mization method (EM) and robust versions of EM in some simulations. We use some results in \citep{Tr98} (see also \citep{Y2004} and \citep{XK10}) for some asymptotic properties. An important fact,  mentioned in \citep{Tr98}, is that  ``It is shown that, using Monte Carlo simulation, the finite sample properties of the ECF estimator are very good, even in the case where the popular maximum likelihood estimator fails to exist\dots''.
We conclude with a further example using real data from the National Institute for Education Evaluation (INEEd) in Uruguay. The dataset is open access, and the codes, written in R, are available on GitHub\footnote{\url{https://github.com/mrleomr/MultivariateMixture}} or can be requested directly from the authors.

In Section~\ref{S:randompartitions} we study the problem of comparing model-based clusterings through the mixture distributions that they induce. This should not be understood as a direct metric on partitions in the sense of the Rand, adjusted Rand or variation-of-information indices, see \citep{Ra71}, \citep{HA85}. Instead, we introduce a model-based distributional discrepancy, computed by aggregating Kolmogorov--Smirnov distances over one-dimensional projections. Throughout that section, the theoretical formulation assumes that both fitted model-based distributions
belong to the same mixture class with a common number $m$ of components. See for instance the reviews by Arabie and Boorman \citep{AB73}, Fowlkes and Mallows  \citep{FM83},  among others.

The way we approach this problem is to compare the fitted probability measures associated with two model-based clustering procedures. This provides a model-based distributional interpretation of the comparison in terms of the underlying probability distributions. If the induced mixture distributions are close, then the corresponding clustering procedures may be regarded as yielding similar model-based representations of the data. This interpretation is distributional and should not be confused with a direct comparison of the resulting partitions.

Algorithms implementing these ideas are described, and then illustrated with some simulations. The simulations indicate a high degree of agreement between several well-known distance-measures and the one that we propose.

Our arguments rely on the notion of a \emph{strong symmetric-matrix uniqueness set}
(\emph{strong sm-uniqueness set}), which provides a finite Cram\'er--Wold system
for Gaussian and $t$-mixtures. We recall the definition and its construction in
Section~\ref{S:mixtures}; see \citep{FMR23a,FMR2025}.

\paragraph{Notation.}
Throughout the manuscript, $m$ denotes the number of mixture components, $d$ the ambient dimension, $N$ the sample size (i.i.d.\ observations), and $k$ the number of projection directions used by the random-projection procedures. In particular, we use $N$ for the sample size to avoid notational conflicts with other indices.

\section{Some preliminaries and recent results to be used}\label{S:mixtures}

In this section, we introduce the fundamental concepts necessary for understanding our work. In particular, we clarify the notion of mixture, and describe a version of the Cram\'er--Wold theorem for mixtures of Gaussian or Student distributions.


\subsection{Mixtures}

Mixture models are fundamental tools in statistics and machine learning, providing a flexible framework for representing heterogeneous data as combinations of simpler component distributions. Since the pioneering work of Pearson \citep{Pe1894}, finite mixture models have been widely used for clustering, density estimation, and latent structure modeling \citep{MP2000,MLR2019}. They have also found applications across many disciplines.

Estimation rates for finite mixture distributions, that is, for the mixture parameters, are also studied in \citep{CH1995,HK2018}.
 A more detailed review of the literature, including specific cases involving Gaussian or Student $t$ mixture distributions, can be found in \citep{MLR2019}.

Let $d \ge 1$ and let  $\cP$ be a family of Borel probability measures on $\RR^d$. A \emph{$\cP$-mixture}  is defined as a convex combination of measures from $\cP$. In other words, it takes the form 
\[
\sum_{j=1}^n\lambda_j P_j
\]
where $P_1,\dots,P_n\in\cP$ and $\lambda_1,\dots,\lambda_n\ge0$ with $\sum_{j=1}^n\lambda_j=1$.

Once the data has been projected, we will focus on univariate distributions. Specifically, if $d=1$, we will consider the two families of univariate distributions: Gaussian distributions and $t$-distributions.  Gaussian distributions are characterized by their densities
\begin{equation}\label{E:Gaussdf}
	f_{\mu,\sigma}(x)=\frac{1}{\sqrt{2\pi\sigma^2}}\exp\Bigl(-\frac{(x-\mu)^2}{2\sigma^2}\Bigr) \quad(x\in\RR).
\end{equation}
A \emph{$t$-distribution} on $\RR$  with $\nu$  degrees of freedom is a Borel measure characterized by the density function
\[
f_{\nu,\mu,\sigma}(x) = c_{\nu,\mu,\sigma} \left(1 + \frac{(x - \mu)^2}{\nu \sigma^2}\right)^{-(\nu + 1)/2},
\]
where $\nu>0$, $\mu\in\RR$ and $\sigma>0$. The constant $c_{\nu,\mu,\sigma}$ is chosen to ensure that 
\[
\int_{\mathbb{R}} f_{\nu,\mu,\sigma}(x)\,dx = 1.
\]
This distribution has a mean of $\mu$  (provided that $\nu>1$) and a variance of $\sigma^2\nu/(\nu-2)$ (if $\nu>2$).

In the multivariate case, where $d > 1$, a Gaussian measure $P$ on $\RR^d$ is defined by a density of the form
\begin{equation}\label{E:density}
	\frac{1}{(2\pi \det(\Sigma))^{1/2}}\exp\Bigl(-\frac{1}{2}(x-\mu)^T\Sigma^{-1}(x-\mu)\Bigr)
	\quad(x\in\RR^d),
\end{equation}
where $\mu\in \RR^d$, and where $\Sigma$ is a real $d\times d$  positive-definite matrix.

A \emph{Gaussian mixture} is a measure on $\RR^d$ that represents a finite convex combination of Gaussian measures. Mixtures of multivariate Gaussian distributions exhibit several advantageous properties. In particular, Titterington et al.\ \citep{TSM85} demonstrate that Gaussian kernel density estimators can approximate any continuous density given a sufficient number of kernels, establishing their universal consistency (see also Scott \citep{Sc92}). Additionally, it is well known that Gaussian mixtures are weak*-dense in the space of all Borel probability measures on $\RR^d$.

Gaussian-mixture models have proven to be highly effective in modeling various real-world data. For a comprehensive examination of their properties, we refer to Titterington et al.\ \citep{TSM85}. These models are flexible and general, with relevant applications documented in numerous fields, including density estimation, machine learning, and clustering. As noted in \citep{CL09}, significant applications include Pearson's early work on modeling crab data \citep{Pe1894}, among many others.

Estimating these models can be quite complex, particularly for high-dimensional data, and typically involves Markov-chain Monte-Carlo methods within a Bayesian framework. Important applications for cluster analysis are discussed in Tadesse, Sha, and Vannucci \citep{TSV05}, as well as in Raftery and Dean \citep{RD06}.

Similarly, when the data exhibit heavy tails, a similar development is plausible using mixtures of Student distributions.
A Student distribution ($t$-distribution) on $\RR^d$ is a measure with density of the form
\[
f_{\nu,\mu,\Sigma}(x)
=c_{\nu,\mu,\Sigma}\Bigl(1+\frac{(x-\mu)^T\Sigma^{-1}(x-\mu)}{\nu}\Bigr)^{-(\nu+d)/2}\quad(x\in\RR^d),
\]
where $\nu>0$, $\mu\in\mathbb{R}^d$, and $\Sigma$ is a positive-definite
$d\times d$ matrix. Once again, the constant $c_{\nu,\mu,\Sigma}$ is chosen to ensure that $\int_{\RR^d}f_{\nu,\mu,\Sigma}(x)\,dx=1$.

\subsection{A Cram\'er--Wold Theorem for Gaussian mixtures and $t$-mixtures}

In this section, we address the problem  for equality between two Gaussian mixtures by analyzing a finite number of projections. The basic Theorem supporting this approach consists of two key components. The first is the abstract result presented in Theorem 4.1 in \citep{FMR2025}. The second is a characterization of Cram\'er--Wold systems for Gaussian measures in $\mathbb{R}^d$ (and, more generally, for elliptical distributions), established in \citep[Theorems~1 and~2]{FMR23a}, which we will now summarize.

Let $S$ be a set of vectors in $\RR^d$.  The corresponding set of lines  $\{\langle x\rangle:x\in S\}$ forms a Cram\'er--Wold system for the Gaussian measures in $\RR^d$ if and only if $S$  satisfies the property that the only real symmetric $d \times d$ matrix $A$ for which $x^TAx=0$ for all $x\in S$ is the zero matrix. A set $S$ with this property is referred to as a \emph{symmetric-matrix uniqueness set} (or \emph{sm-uniqueness set} for short).

In \citep{FMR23a}, it was demonstrated that an sm-uniqueness set for $\RR^d$ must span $\RR^d$and contain at least $(d^2+d)/2$  vectors. We will denote $S$ as a \emph{strong sm-uniqueness set} if every subset of $S$ containing $(d^2+d)/2$ vectors is also an sm-uniqueness set.

We now recall  the Cram\'er--Wold Theorem for Gaussian mixtures   \citep[Theorem~4.1]{FMR2025}.

\begin{theorem}\label{T:gaussmixture}
	Let $P$ and $Q$ be convex combinations of  $m$ Gaussian measures on $\RR^d$ respectively. 
	Let $S$ be a strong sm-uniqueness set for $\RR^d$ containing at least $(1/2)(2m-1)(d^2+d-2)+1$ vectors.
	If $P_{\langle x\rangle}=Q_{\langle x\rangle}$ for all $x\in S$, then $P=Q$.
\end{theorem}

We also recall an analogue of Theorem~\ref{T:gaussmixture} for mixtures of  multivariate $t$-dist\-ributions, given in  \citep[Theorem~4.2]{FMR2025}, thereby allowing heavy-tailed distributions. 

\begin{theorem}\label{T:tmixture}
	Let $P$ and $Q$ be convex combinations of  $m$
	multivariate  $t$-distributions on $\RR^d$ respectively. 
	Let $S$ be a strong sm-uniqueness set for $\RR^d$ containing at least $(1/2)(2m-1)(d^2+d-2)+1$ vectors.
	If $P_{\langle x\rangle}=Q_{\langle x\rangle}$ for all $x\in S$, then $P=Q$.
\end{theorem}


\subsection{Strong sm-uniqueness sets}

Theorems~\ref{T:gaussmixture} and \ref{T:tmixture}  beg the question as to whether there exist strong sm-uniqueness sets of arbitrarily large cardinality. An affirmative  answer is given in the following result based on \citep[Theorem~4.3]{FMR2025}, which also  provides a realistic method for generating them.

We next recall the notion of a (strong) sm-uniqueness set, which is needed for the theorem below.

\begin{definition}[sm-uniqueness and strong sm-uniqueness sets {\cite{FMR23a,FMR2025}}]
\label{def:ssmu}
Let $S\subset\RR^d$. We say that $S$ is an \emph{sm-uniqueness set} if the only real
symmetric $d\times d$ matrix $A$ satisfying $x^\top A x = 0$ for all $x\in S$ is $A=0$.
If $|S|\ge d(d+1)/2$, we say that $S$ is \emph{strong} if every subset of $S$ with
cardinality $d(d+1)/2$ is itself an sm-uniqueness set.
\end{definition}

\begin{theorem}\label{T:ssmu}
	Let $d\ge2$, let $k\ge (d^2+d)/2$, and let $v_1,\dots,v_k$ be independent random vectors in $\RR^d$
	whose distributions are given by densities on $\RR^d$. Then, with probability one, the set $\{v_1,\dots,v_k\}$ 
	is a strong sm-uniqueness set for $\RR^d$.
\end{theorem}

This Theorem allows us to easily generate strong sm-uniqueness sets, by uniformly sampling $ k $ random directions on the unit sphere. With these results in hand we are ready to study the two statistical problems mentioned above.

\section{A new estimator for Gaussian mixtures}\label{S:GMestimator}

\subsection{Introduction}

In this section, we shall develop a new estimator of a multivariate mixture of Gaussian distributions based on two ideas: random projections to obtain one-dimensional mixtures of normal distributions, and one-dimensional estimators of the mixture based on characteristic functions (see \citep{Tr98} and \citep{XK10} for the second idea). We shall compare our random-projection method (RP) with the expectation-maximization (EM) algorithm.

There is a large literature on this subject. See for instance the well-known work in \citep{MV2010}, the recent work \citep{DWYZ23} and the references therein. The article \citep{MV2010} considers general Gaussian multivariate models being $\epsilon$-statistically learnable, that is, $F= \sum_{i=1}^n \lambda_i F_i$  satisfying $\min_i \lambda_i \geq \epsilon$ and $\min_{i \neq j} \frac{1}{2} \int \vert f_i - f_j \vert \geq \epsilon$. In \citep{DWYZ23}, very strong results are shown for the location case, where the covariances are all equal and known.    In contrast, we just prove strong consistency, but this is done for the general case where the covariances are unknown and may be different, in an almost universal framework assuming only a differentiability condition on the characteristic functions, stated before  \citep[Theorem~4.1]{FMR2025} below.


\subsection{Description of the algorithm} \label{S:algo}

We propose a two-step sequential estimation procedure based on random projections and the empirical characteristic function (ECF). 

Let $P$ be a Gaussian mixture on $\RR^d$, with density 
\[
\sum_{j=1}^m{\lambda_j}\frac{1}{(2\pi \det(\Sigma_j))^{1/2}}
\exp\Bigl(-\tfrac{1}{2}(x-\mu_j)^T\Sigma_j^{-1}(x-\mu_j)\Bigr),
\quad(x\in\RR^d),
\]
where $\lambda_1,\dots,\lambda_m\in[0,1]$, $\sum_{j=1}^m\lambda_j=1$,
$\mu_1,\dots,\mu_m\in\RR^d$, and $\Sigma_1,\dots,\Sigma_m$ are $d\times d$ positive-definite matrices.
We denote $\Lambda:=(\lambda_1,\dots,\lambda_m)$, $\bm{\mu}:=(\mu_1,\dots,\mu_m)$,
$\bm{\Sigma}:=(\Sigma_1,\dots,\Sigma_m)$, and $\Theta:=(\Lambda,\bm{\mu},\bm{\Sigma})$.

Given a unit vector $u\in\RR^d$ and $\theta\in\Theta$, 
$P_{\langle u\rangle}(\theta)$ denotes the projection of the Gaussian mixture with parameters $\theta$ in the direction $u$.  
Given an iid sample $X_1,\dots,X_N$, let $P_{N,\langle u\rangle}$ be the projection of the empirical measure in the same direction.

In the first step, the mixture weights $\Lambda$ are estimated by fitting univariate projected mixtures along several random directions. This exploits the fact that linear projections of a Gaussian mixture are again Gaussian mixtures, with transformed means and variances. The estimation is performed by minimizing the discrepancy between the empirical and theoretical characteristic functions of the projected data. Averaging across multiple random directions yields a global estimator of the weights.

In the second step, with the weights fixed, the projected means and variances are re--estimated. From these, the original mean vectors and covariance matrices of the mixture are reconstructed by solving quadratic optimization problems. 

The overall procedure is summarized in Algorithm~\ref{alg:twostep}.

In what follows, we write $\mathbb{S}^{d-1}$
for the unit sphere in $\RR^d$.
For a direction $u\in \mathbb{S}^{d-1}$ and component $j=1,\dots,m$, we write
\[
\mu^{\mathrm{proj}}_{u,j}:=\langle u,\mu_j\rangle,
\qquad
\tau^2_{u,j}:=\langle u,\Sigma_j u\rangle,
\]
for the projected mean and projected variance of the $j$th component along
direction $u$. Let
\[
Q_u\!\left(\mu^{\mathrm{proj}}_u,\tau_u^2,\Lambda\right)
=
\Bigl(
\widehat Z_u - Z_u\!\left(\mu^{\mathrm{proj}}_u,\tau_u^2,\Lambda\right)
\Bigr)^\top
W
\Bigl(
\widehat Z_u - Z_u\!\left(\mu^{\mathrm{proj}}_u,\tau_u^2,\Lambda\right)
\Bigr),
\]
where $\widehat Z_u$ denotes the empirical vector and
$Z_u(\mu^{\mathrm{proj}}_u,\tau_u^2,\Lambda)$ the corresponding model-implied
vector introduced in Appendix~\ref{apend1}.

\begin{algorithm}[H]
	\caption{Two-step estimation via random projections and ECF/GMM with alignment}
	\label{alg:twostep}
	\begin{algorithmic}[1]
		\Require Integer $m$; dimension $d$; sample $X_1,\dots,X_N$; number of directions $k$
		
		\State Draw $k$ random directions $u_1,\dots,u_k\sim \mathrm{Unif}(\mathbb{S}^{d-1})$
		
	\textbf{Step 1:} 	\For   {$r=1$ \textbf{to} $k$}
		\State Estimate $\bigl(\widehat{\mu}^{\mathrm{proj}}_{u_r,(1)},\widehat{\tau}^{2}_{u_r,(1)},\widehat{\Lambda}_{u_r,(1)}\bigr)$ by
		\State \quad $\arg\min_{\mu^{\mathrm{proj}}_{u_r},\,\tau^{2}_{u_r},\,\Lambda}
		Q_{u_r}\!\left(\mu^{\mathrm{proj}}_{u_r},\tau^{2}_{u_r},\Lambda\right)$
		\EndFor
		
		\State \textbf{Alignment:} choose a pivot direction $u_{r^\ast}$ (for instance, one maximizing the minimum pairwise separation among the estimated projected means)
		\State Order its components by $\widehat{\mu}^{\mathrm{proj}}_{u_{r^\ast},(1),j}$ and treat the ordered triples
		\State \quad $\Bigl(\widehat\Lambda_{u_{r^\ast},(1),j},\widehat{\mu}^{\mathrm{proj}}_{u_{r^\ast},(1),j},\widehat\tau^2_{u_{r^\ast},(1),j}\Bigr)$
		\State as prototypes
		
		\For{$r=1$ \textbf{to} $k$}
		\State Solve an assignment problem matching the triples
		\State \quad $\Bigl(\widehat\Lambda_{u_r,(1),j},\widehat{\mu}^{\mathrm{proj}}_{u_r,(1),j},\widehat\tau^2_{u_r,(1),j}\Bigr)$
		\State to the pivot prototypes, obtaining aligned estimates
		\State \quad $\Bigl(\widetilde\Lambda_{u_r,(1),j},\widetilde{\mu}^{\mathrm{proj}}_{u_r,(1),j},\widetilde\tau^2_{u_r,(1),j}\Bigr)$
		\EndFor
		
		\State Compute $\widehat\Lambda_{(1)}\gets \frac1k\sum_{r=1}^k \widetilde\Lambda_{u_r,(1)}$
		
		\textbf{Step 2:}
		
		\State  Fix $\widehat\Lambda_{(1)}$ and re-estimate
		\State \quad $\left\{\bigl(\widehat{\mu}^{\mathrm{proj}}_{u_r,(2)},\widehat\tau^2_{u_r,(2)}\bigr)\right\}_{r=1}^k$
		\State as described in Appendix~\ref{apend1}
		
		\For{$j=1$ \textbf{to} $m$}
		\State Reconstruct $\widehat\mu_j$  $\widehat\Sigma_j$ and by

\begin{align}
	\label{E:eqmu}
	\hat{\mu}_j &:= \arg\min_{\mu_j\in\mathbb{R}^d}\sum_{r=1}^k
		\left(u_r^\top \mu_j-\widehat{\mu}^{\mathrm{proj}}_{u_r,(2),j}\right)^2  \\ 
	\label{E:eqsigma}
	\hat{\Sigma}_{j} &:= \arg\min_{\Sigma_j\in  \sym_d^+(\mathbb{R})}\sum_{r=1}^k
		\left(u_r^\top \Sigma_j u_r-\widehat\tau^2_{u_r,(2),j}\right)^2,
\end{align}
with $ \sym_d^+(\mathbb{R})$ is the cone of $d\times d$ positive semidefinite matrices.

		\State \Return $\bigl(\widehat\Lambda_{(1)},\widehat\mu,\widehat\Sigma\bigr)$
		\EndFor
	\end{algorithmic}
\end{algorithm}

In Algorithm \ref{alg:twostep}, the parameters are initialized separately in each random one-dimensional projection, and the main driver of the initialization is a $k$-means split with $k=m$. Specifically, for a randomly sampled direction $u$ (normalized to unit length), the data are projected as $Y_i=u^\top X_i$, and we run $k$-means with $m$ clusters on $\{Y_i\}_{i=1}^N$. The resulting $m$ clusters provide a data-driven starting point: the initial projected locations are set to the within-cluster sample means, the projected scales to the within-cluster sample variances (truncated below by a small constant for numerical stability), and the initial mixing weights are taken as the empirical cluster proportions.

To avoid poor starts, we retain only those projections for which all clusters contain at least a prescribed minimum number of observations and the minimum pairwise separation between the projected $k$-means centroids exceeds a fixed multiple of $\mathrm{sd}(Y)$. After ordering components by increasing projected mean, the $k$-means-based initialization may be optionally refined by a single
soft-EM update for the mixing weights and then passed to the ECF/GMM estimation step in that projection.

The reconstruction of each covariance matrix $\Sigma_j$ from the projected
variances is obtained by minimizing the quadratic expression
\[
  f(\Sigma) = \sum_{r=1}^k \bigl(u_r^\top \Sigma u_r - s_{r,j}\bigr)^2,
  \qquad \Sigma \in \sym_d(\mathbb{R}),
\]
where $\sym_d(\mathbb{R})$ denotes the space of real symmetric $d\times d$  matrices.
This objective is always convex in $\Sigma$, but uniqueness of the solution
requires strict convexity, which in turn depends on the geometry of the
directions $\{u_r\}$. The following proposition makes this condition explicit.

\begin{proposition}[Uniqueness of covariance reconstruction]\label{P:uniqueness}
Let $u_1,\dots,u_k \in \mathbb{S}^{d-1}$ and consider
\[
  f(\Sigma) = \sum_{r=1}^k \bigl(u_r^\top \Sigma u_r - s_r\bigr)^2,
  \qquad \Sigma \in \sym_d(\mathbb{R}).
\]
If $\{u_r u_r^\top : r=1,\dots,k\}$ spans $\sym_d(\mathbb{R})$, equivalently
if $\mathrm{rank}(A) = d(d+1)/2$ where the $r$-th row of $A$ is
$\mathrm{vecsym}(u_r u_r^\top)$, then $f$ is strictly convex on $\sym_d(\mathbb{R})$
and admits a unique minimizer. In particular, the 
equation \eqref{E:eqsigma} has a unique solution under the rank condition.
\end{proposition}

\begin{proof}[Sketch of proof]
Identifying $\sym_d(\mathbb{R})$ with $\mathbb{R}^{q}$, $q=d(d+1)/2$, via
$\theta=\mathrm{vecsym}(\Sigma)$, we can write
\[
  f(\Sigma) = \bigl\|A\,\theta - s\bigr\|_2^2,
\]
with $A\in\mathbb{R}^{k\times q}$ having $r$-th row $\mathrm{vecsym}(u_r u_r^\top)$.
The spanning condition $\{u_r u_r^\top\}$ spans $\sym_d(\mathbb{R})$ is equivalent to
$\mathrm{rank}(A)=q$, so $A^\top A$ is positive definite. Hence the Hessian of
$f$ with respect to $\theta$ is $2A^\top A\succ 0$, and $f$ is strictly convex,
with a unique minimizer $\theta^\star$, hence a unique $\Sigma^\star$.
\end{proof}

Further details are provided in Appendix~\ref{apend1}.

\begin{remarks}

	(i)  If $\Lambda$ is known or estimated by another mechanism,
	then the estimation procedure starts at Step~2. For example, as in \citep{dBCMM08}, this happens when one starts with a cluster procedure that has already been performed. This allows to have a $\Lambda$-estimator across the group sizes obtained.  
	
	(ii)   Robust estimators can be obtained by replacing the $L^2$-minimization by the $L^1$-one, that is, by solving the equations
	\[
	\hat{\mu}_j = \argmin_{\mu_j} \sum_{r=1}^k \bigl|\langle u_r,\mu_j\rangle - \widehat{\mu}^{\mathrm{proj}}_{u_r,j} \bigr|, \ \ \hat{\Sigma}_{j}= \argmin_{\Sigma_j \in \sym_d^+(\mathbb{R})} \sum_{r=1}^k   \bigl| \langle u_r,\Sigma_j u_r\rangle - \widehat{\tau}^2_{u_r,j} \bigr|,
	\]
	for all $j=1,2, \dots, m$,
	where $\sym_d^+(\mathbb{R})$ denotes the cone of positive semi-definite real $d\times d$ matrices.

	Alternatively,  in the spirit of the proposed median of means, see for instance \citep{DLLO16} and the references therein,
	one can proceed as as follows.
	We split the sample into $L$ disjoint subsamples. For each subsample and $j$ we calculate $\hat \mu_{j,\ell}$ and $\hat \Sigma_{j,\ell}$, $j=1, \dots m$, $\ell=1, \ldots, L$ according to equations (\ref{E:eqmu}) and (\ref{E:eqsigma}),
	and we take the median of the corresponding $L$ values obtained for each $j$. Then we proceed as before.

(iii)  The argmin in (\ref{E:eqmu}) can be solved by ordinary least squares. For the constrained optimization problem in (\ref{E:eqsigma}), we work over the cone $\sym_d^+(\mathbb{R})$ of symmetric positive semidefinite $d\times d$ matrices and use a primal--dual
interior-point method (IPM), which is standard in semidefinite programming. It is important to distinguish existence from uniqueness. The objective
\[
f(\Sigma)=\sum_{r=1}^k \left(u_r^\top \Sigma u_r-\widehat\tau^2_{u_r,(2),j}\right)^2
\]
is convex and continuous on $ \sym_d^+(\mathbb{R})$, so, (\ref{E:eqsigma}) is a convex optimization
problem. However, convexity alone does not imply uniqueness. As established in
Proposition~1, uniqueness of the minimizer holds under the spanning/rank
condition that $\{u_r u_r^\top:r=1,\dots,k\}$ span $\sym_d(\RR)$. If one considers instead the barrier-penalized objective
\[
f_t(\Sigma)=t\,f(\Sigma)-\log\det(\Sigma),
\]
defined on the interior of $ \sym_d^+(\mathbb{R})$, then for each fixed $t>0$ the barrier term enforces strict convexity on the feasible region, and the penalized problem admits a unique minimizer. This uniqueness pertains to the barrier
problem and should be distinguished from uniqueness of the original problem \ref{E:eqsigma}. See, for example, \citep{BV2004,RW2009}.

(iv) While semidefinite programming (SDP) is computationally efficient for small-to-moderate dimensions, various alternatives are available for high-dimensional settings. Examples include first-order methods such as the Alternating Direction Method of Multipliers (ADMM), low-rank matrix approximations when the covariance structure allows, and the exploitation of sparsity or block structure in the projection matrices to reduce computational cost, see for example \citep{JBAS2010} and \citep{WGY2010}.

(v) In general, the problem becomes increasingly challenging as the dimension of the space grows. The parameter space expands substantially when allowing for general covariance matrices, which increases complexity and renders the algorithm more unstable, since one must optimize over symmetric positive definite matrices in addition to the weights and means. Recently, in \citep{YWR2024}, a novel and very general problem is studied that includes the more complex setup of  compound probability distributions under the restriction of identical and isotropic covariances, which for our setting requires that all  covariance matrices be the identity.

(vi)    We note that the algorithm can also be applied to mixtures of multivariate $t$-distributions. In this case, the degrees-of-freedom parameter $\nu$ is assumed to be fixed, known, and common to all mixture components, while the covariance matrices in the Gaussian formulation are replaced by the corresponding scatter matrices.

(vii) It is important to note that Step~1 can be implemented with any other consistent univariate-mixture estimators at each direction $u_r$,
for instance the univariate EM algorithm, instead of the ECF-based criterion.
\end{remarks}

\subsection{A consistency result}

The results in \citep{Tr98} and \citep{XK10} that we will use below are based on a general theorem in \citep{He77},
where the only assumption is that the characteristic function is differentiable. 
More precisely, it is assumed that the regular case holds, namely that $I_n(t)$ can be differentiated under the integral sign.
Proposition~1 in \citep{Tr98}, which is based on the work in \citep{FM1981a}, \citep{FM1981b}, can be stated as follows.

\begin{proposition}
Let $\{t_1, \ldots, t_M\}$ be distinct fixed grid points. Then the estimator $\hat \theta$ of $\theta$   is strongly consistent and asymptotically normally distributed, with a given covariance matrix.
\end{proposition} 

How to choose the sequence $\{t_1, \ldots, t_M\}$ as well as the size $M$ in an optimal way for a given sample size $N$ is still an open problem. However, in \citep{FM1981b} it is suggested that the sequence should be taken equally spaced, 
namely $t_j = \tau j, \ j=1, \ldots M$ for $\tau \in \mathbb R^+$.

\begin{theorem} \label{teoconsis}
	Under the weak assumption that,
	\[
	I'_n(\theta)= \int \frac{\partial }{\partial \theta}\vert \hat \psi_n(t)- \psi(t, \theta)\vert ^2 dP(t),
	\]
	the estimators $\hat {\Lambda}_{(1)},\hat{\mu}_j,\hat{\Sigma}_j$
	for $j=1,\dots, m$ are strongly consistent.
\end{theorem}

\begin{proof}
	First observe that, from the results in \citep{Tr98} and \citep{XK10}
	for the finite set of estimators for the projected one-dimensional projections, 
	we can derive the strong consistency of $\hat \Lambda_{(1)}$.  
	See \citep[Proposition~1]{Tr98}, where strong consistency and asymptotic normality are established. 
	
	For the next step, using the value of $\hat \Lambda_{(1)}$,
	we consider the estimators of $\hat{\mu}_j$ and $\hat{\Sigma}_j$
	for $j=1,\dots m$ given by the least-squares equations (\ref{E:eqmu}) and (\ref{E:eqsigma}). 
	Since 
	\[
	\sum_{r=1}^k \bigl(\langle u_r,\mu_j\rangle - \mu^{\mathrm{proj}}_{u_r,j} \bigr)^2 =0
	\quad\text{and}\quad
	\sum_{r=1}^k  \bigl(\langle u_r,\Sigma_j u_r\rangle - \tau^2_{u_r,j} \bigr)^2=0
	\]
	for the true values of $\mu_j$ and $\Sigma_j~( j=1,\dots m)$,
	we have
	\[
	\hat{\mu}_j  
	= \argmin_{\mu_j} \sum_{r=1}^k \Bigl(\bigl(\langle u_r,\mu_j\rangle - \widehat{\mu}^{\mathrm{proj}}_{u_r,j}\bigr)^2 -\bigl(\langle u_r,\mu_j\rangle - \mu^{\mathrm{proj}}_{u_r,j}\bigr)^2 \Bigr) \to 0\quad\text{a.s.}
	\]
	and
	\[
	\hat{\Sigma}_{j} =  \argmin_{\Sigma_j \in \sym_d^+(\mathbb{R})} \sum_{r=1}^k \Bigl(\bigl(\langle u_r,\Sigma_j u_r\rangle - \widehat{\tau}^2_{u_r,j}\bigr)^2 - \bigl(\langle u_r,\Sigma_j u_r\rangle - \tau^2_{u_r,j}\bigr)^2 \Bigr) \to 0 \quad\text{a.s.}
	\]
	for all $j=1,\dots, m$. This concludes the proof. 
\end{proof}

The only assumption in our Theorem \ref{teoconsis} is that $I_n(t)$ can be differentiated under the integral sign.
As mentioned in \citep[p. 258]{He77}, the only random variables excluded are those with characteristic functions vanishing outside an interval depending on the parameter $\theta$ and those for which $\psi$ is not differentiable with respect to $\theta$. These conditions are fulfilled for the univariate mixtures that we have under consideration. Indeed, the differentiation condition follows for instance from \citep{P1956}.
	
Therefore our strong consistency result is universal, i.e., it holds without further assumptions for Gaussian or $t$-mixtures. Moreover, the theorem-level requirement in \citep{FMR2025} shows that it suffices to take
\[
k \;\ge\; \frac{1}{2}(2m-1)(d^2+d-2)+1
\]
projection directions, where $m$ denotes the number of mixture components and $d$ the ambient dimension. Hence the required number of directions grows linearly in $m$ and quadratically in $d$.

No restriction is placed on the sample size.

An interesting related but different approach is given by Moitra and Valiant in \citep{MV2010}  for the estimation of the parameters of mixture of multivariate Gaussians (GMM), where they provide an algorithm that has a polynomial running time. 
To state their main result, we provide some definitions. In what follows, we write $[k]:=\{1,2,\dots,k\}$.

\begin{definition}
(i) Given two $d$-dimensional GMMs of $m$ Gaussians, 
$F=\sum_i w_i N(\mu_i, \Sigma_i)$ and $\hat F= \sum_i \hat w_i N(\hat\mu_i, \hat\Sigma_i)$,
we call $\hat F$ an \emph{$\epsilon$-close estimate} for $F$ 
if there is a permutation function $\pi: [k] \to [k]$ such that, for all $i \in [k]$,
\[
\|w_i - \hat w_{\pi[i]}\| \le \epsilon
\quad\text{and}\quad
D(N(\mu_i, \Sigma_i), N(\hat \mu_{\pi(i)}, \hat \Sigma_{\pi(i)} )) \le \epsilon.
\]
(ii) We call a GMM $F= \sum_i w_i F_i$ \emph{$\epsilon$-statistically learnable} if both $\min_i w_i \ge \epsilon$ and $\min_i D(F_i, F_j)\ge \epsilon$.
\end{definition}

The result of Moitra and Valiant is given in the following theorem.
 
 \begin{theorem}[\protect{\cite[Theorem~1]{MV2010}}]
 Given an $d$-dimensional mixture of $m$ Gaussians $F$ that is $\epsilon$-statistically learnable, there is an algorithm that, with probability at least $1-\delta$, outputs an $\epsilon$-close estimate $\hat F$, and the running-time and data requirements of the algorithm (for any fixed $m$) are polynomial in $n$, $1/\epsilon$ and $1/\delta$.
 \end{theorem}
  
 If one is only interested on $D(F, \hat F)$, then in \citep[Corollary 2]{MV2010} there is stated a similar result without the restriction that $F$ be $\epsilon$-statistically learnable.

Sharp results for the estimation of univariate finite mixtures are also provided
by Heinrich and Kahn in \citep{HK2018}, where   they study the rates of convergence of the parameters of the mixture,  and,  under some regularity and strong identifiability conditions, find  the optimal local minimax rate of estimation of a univariate mixture with $m$ components. However they do not address the multivariate case.  

	\subsection{Simulations} 
	
	\subsubsection{Example 1.}
	
	Using the above algorithm, we estimate the parameters of a mixture of two bivariate $t$-distributions $t(\nu,\mu,\Sigma)$, where $\nu$, $\mu$, and $\Sigma$ denote the degrees of freedom, location, and scale matrix parameters, respectively. Consider
\[
\begin{aligned}
F &:= \lambda_1\, t(\nu, \mu_{1}, \Sigma_{1}) 
     + (1-\lambda_1)\, t(\nu, \mu_{2}, \Sigma_{2}), \\[0.3em]
\mu_{1} &= (0,0), \quad \mu_{2} = (\eta,0), \\[0.3em]
\Sigma_{1} &= \begin{pmatrix} 1 & 0 \\ 0 & 1/2 \end{pmatrix}, \quad
\Sigma_{2} = \begin{pmatrix} 1/2 & 0 \\ 0 & 1 \end{pmatrix}, \\[0.3em]
\lambda_1 &= 0.3, \quad \nu = 4.
\end{aligned}
\]
Based on $200$ i.i.d.\ draws from $F$ (see Figure~\ref{F:mez}), our goal is to estimate $\mu_{1}$, $\mu_{2}$, $\Sigma_{1}$, and $\Sigma_{2}$. We implement the proposed algorithm by projecting the data onto $k=50$ randomly selected directions. This experiment is replicated $100$ times. In addition, we consider four separability scenarios by letting $\eta \in \{1/2,\,1,\,3/2,\,2\}$, which progressively increases the separation between the component locations.

	\begin{figure}[tbp]
		\centering
		\subfloat{\includegraphics[width=140mm]{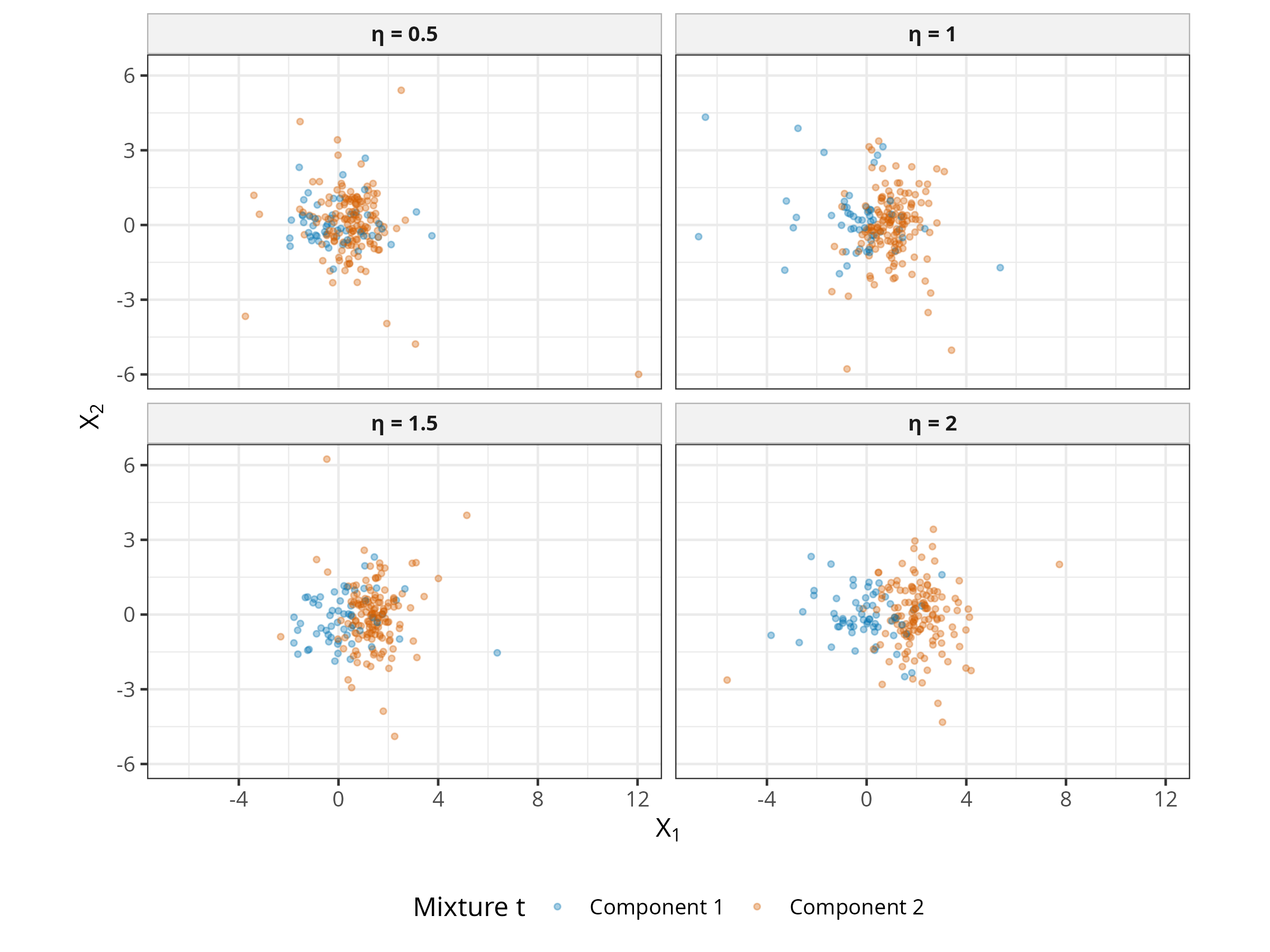}}
\caption{Bivariate $t$-mixture samples for increasing separation $\eta$ (four panels: $\eta \in \{1/2,1,3/2,2\}$).}
		\label{F:mez}
	\end{figure}

	In this Example, we present a small  comparative study of the estimation derived by the random projections method (RP) with respect to the expectation-maximization method for mixtures of t-distributions (we denote \textit{EM-st}).  This last method, developed in \citep{PM00},  is implemented by means of the function \textit{fmmst} of the \textit{EMMIXuskew} package of the R language. In this case the skew parameter was prefixed to $0$.
	
{Table~\ref{Tb:mixla_eta} reports the average parameter estimates across the replicates, together with their corresponding standard deviations. In addition, we report the (replicate-averaged) confusion matrices (see Table \ref{Tb:confusion_all1}) computed over the $100$ Monte-Carlo runs---comparing the true component labels with the posterior (MAP) allocations obtained after parameter estimation, for both the EM-st and the RP procedures. These confusion matrices are provided in Appendix~\ref{A:confusion}.

In each replicate, we compute estimation errors by comparing the estimated parameters with their true values. Specifically, we construct boxplots of (i) the $L^2$-error for the mixing weight $\lambda_1$, (ii) the $L^2$-error for the mean vectors, and (iii) the Frobenius-distance for the scatter matrices. These boxplots are reported for each separability scenario $\eta \in \{1/2,1,3/2,2\}$; see Figure~\ref{F:SS2} for EM-st and RP, respectively. Overall, both algorithms exhibit comparable performance across scenarios. In Figure~\ref{F:SS1}, the boxplots show that the estimation errors for the mixing coefficients $\lambda$ are similar for both methods.

	\begin{figure}[tbp]
		\centering
		\subfloat{\includegraphics[width=57mm]{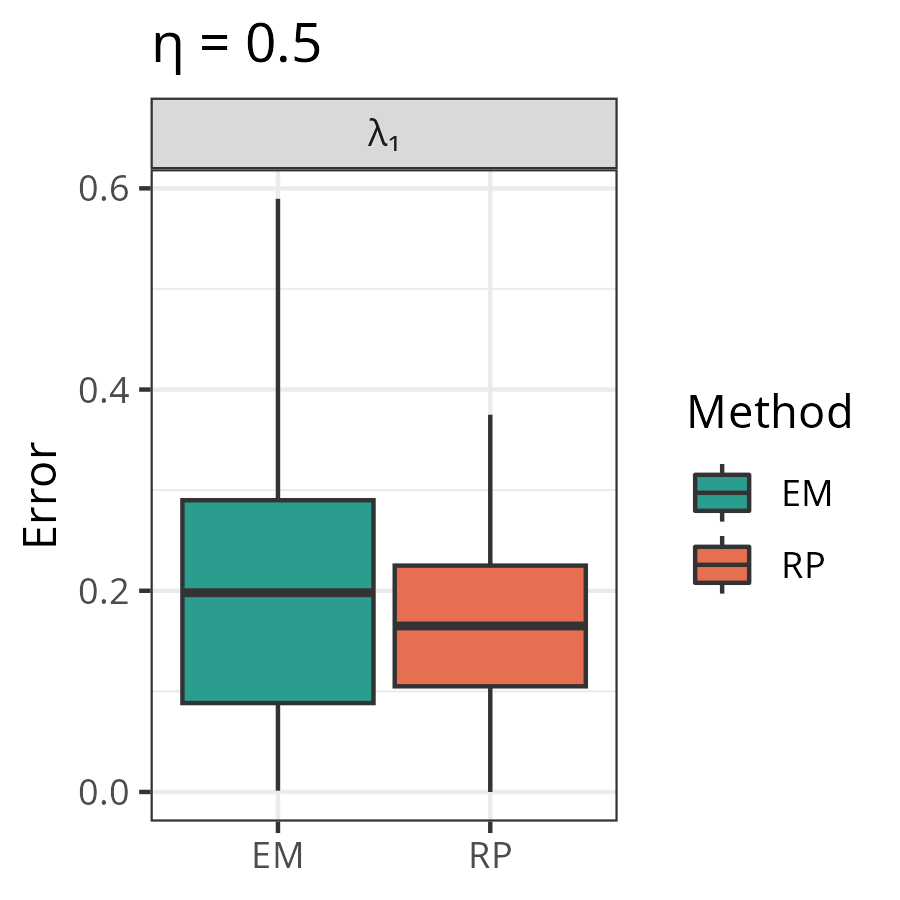}}
		\subfloat{\includegraphics[width=57mm]{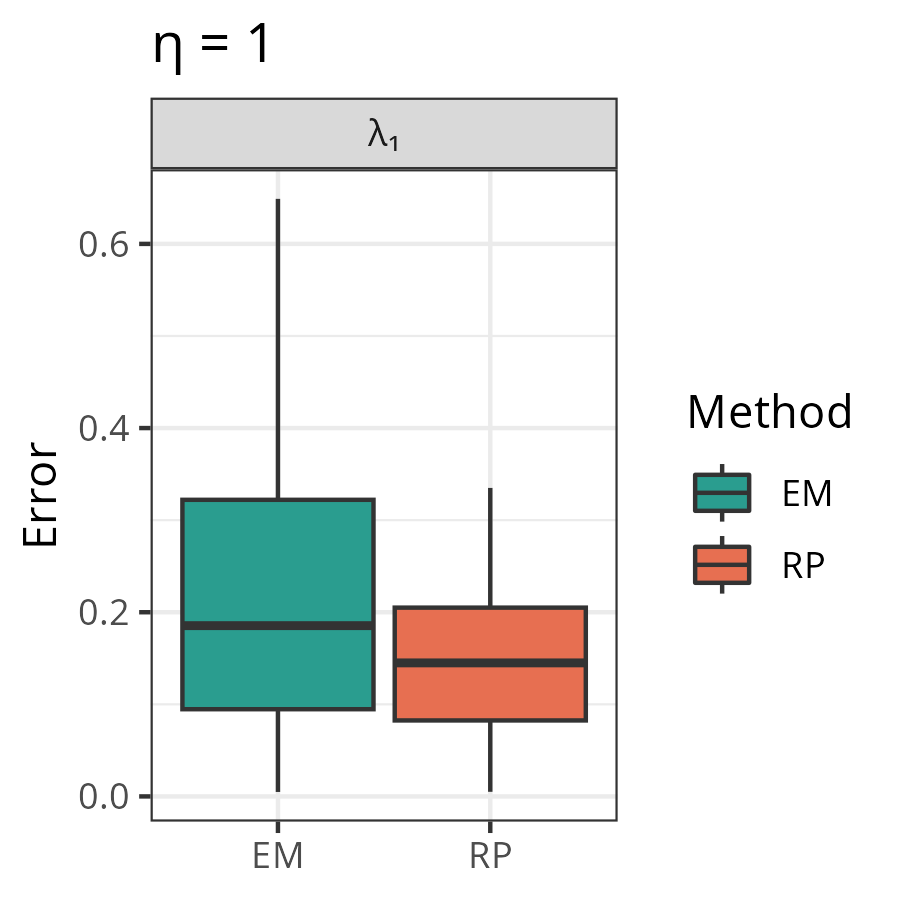}}
		\\
		\subfloat{\includegraphics[width=57mm]{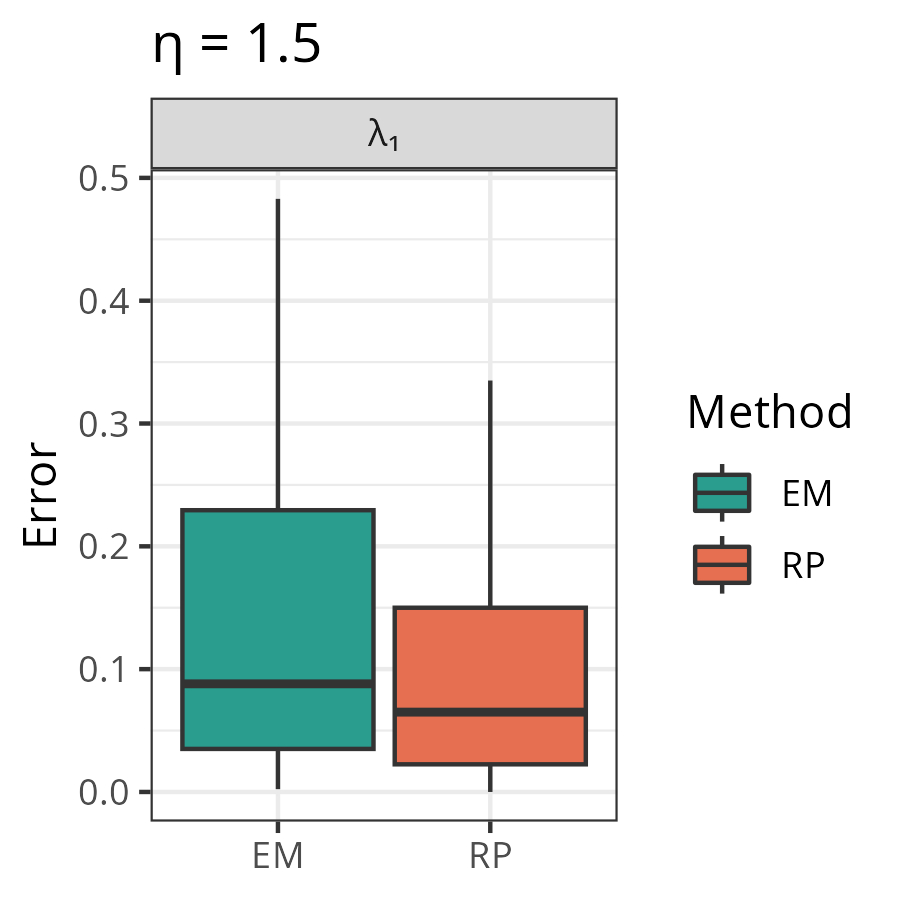}}
		\subfloat{\includegraphics[width=57mm]{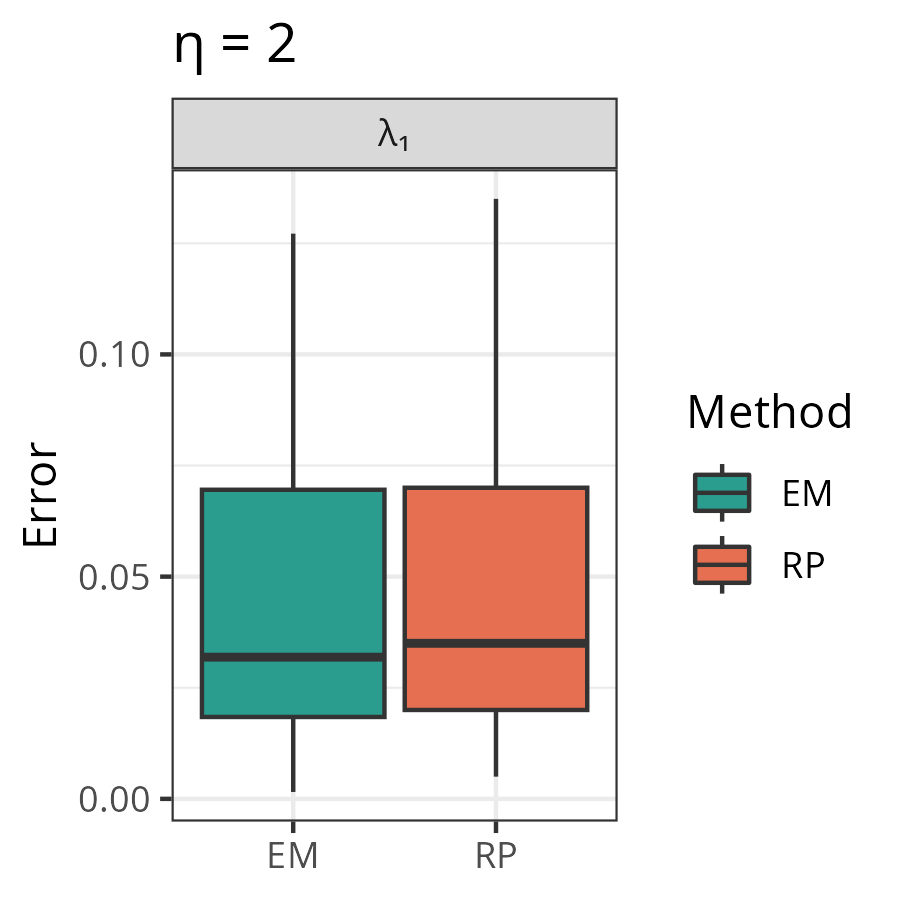}}
\caption{Mixing-weight error for $\lambda_1$: EM-st vs RP across separability scenarios $\eta \in \{1/2,1,3/2,2\}$ (boxplots over 100 replicates).}
		\label{F:SS1}
	\end{figure}

	\begin{figure}[tbp]
		\centering
		\subfloat{\includegraphics[width=67mm]{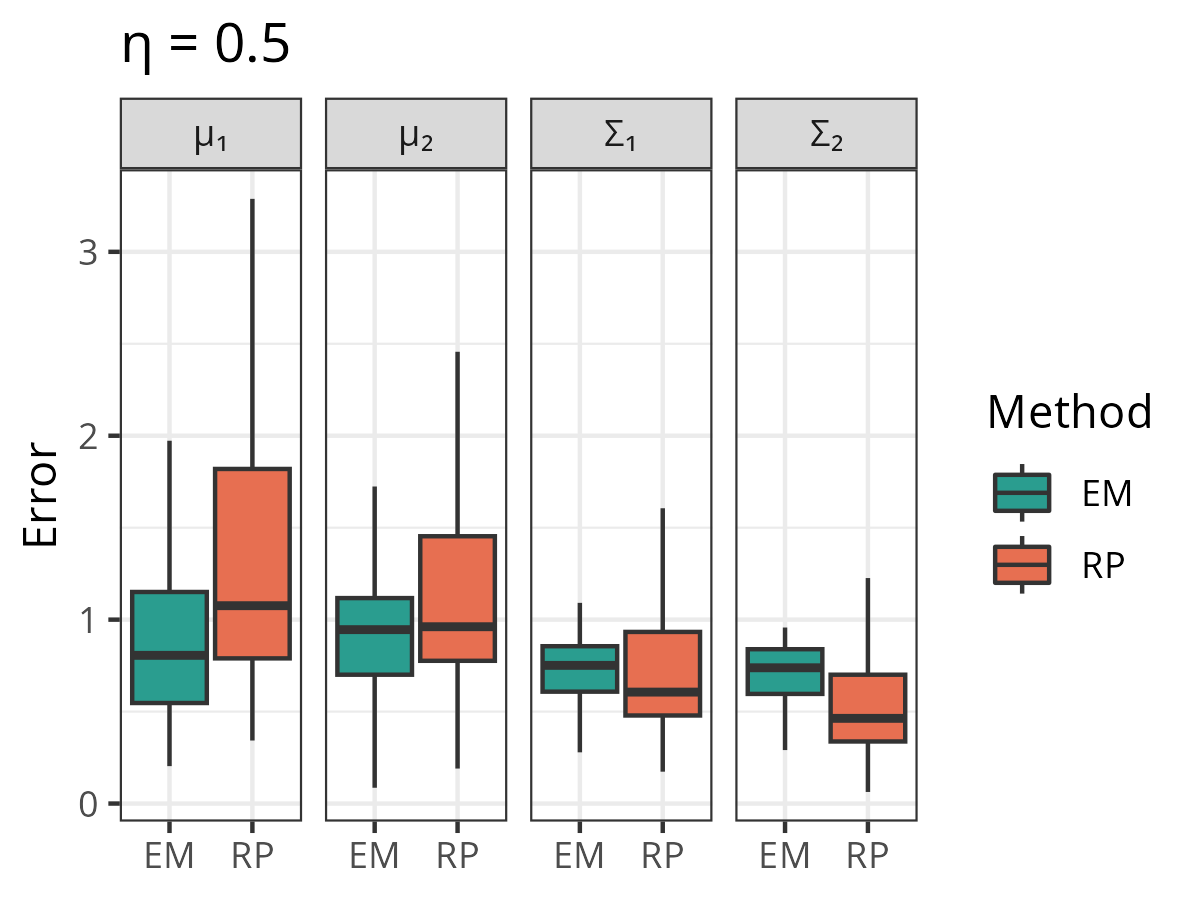}}
		\subfloat{\includegraphics[width=67mm]{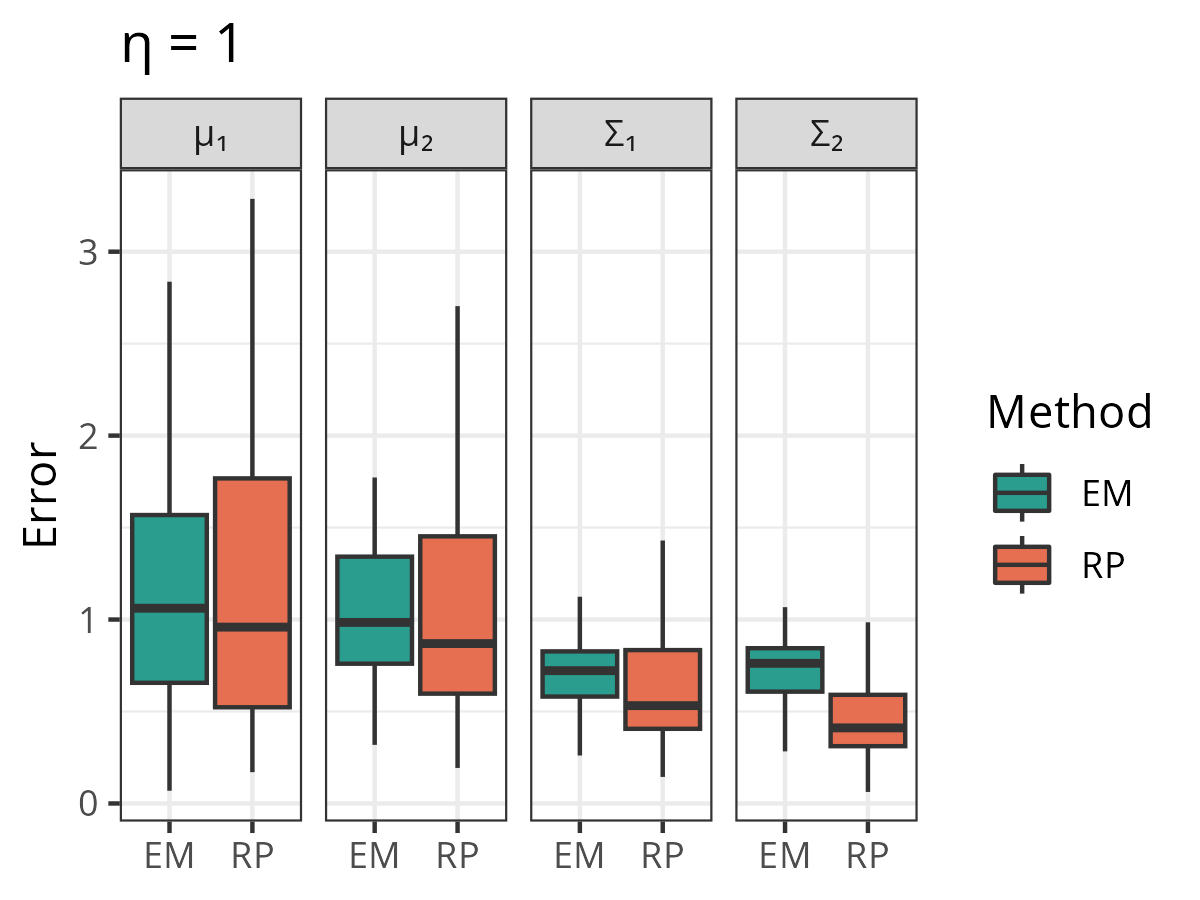}}
		\\
		\subfloat{\includegraphics[width=67mm]{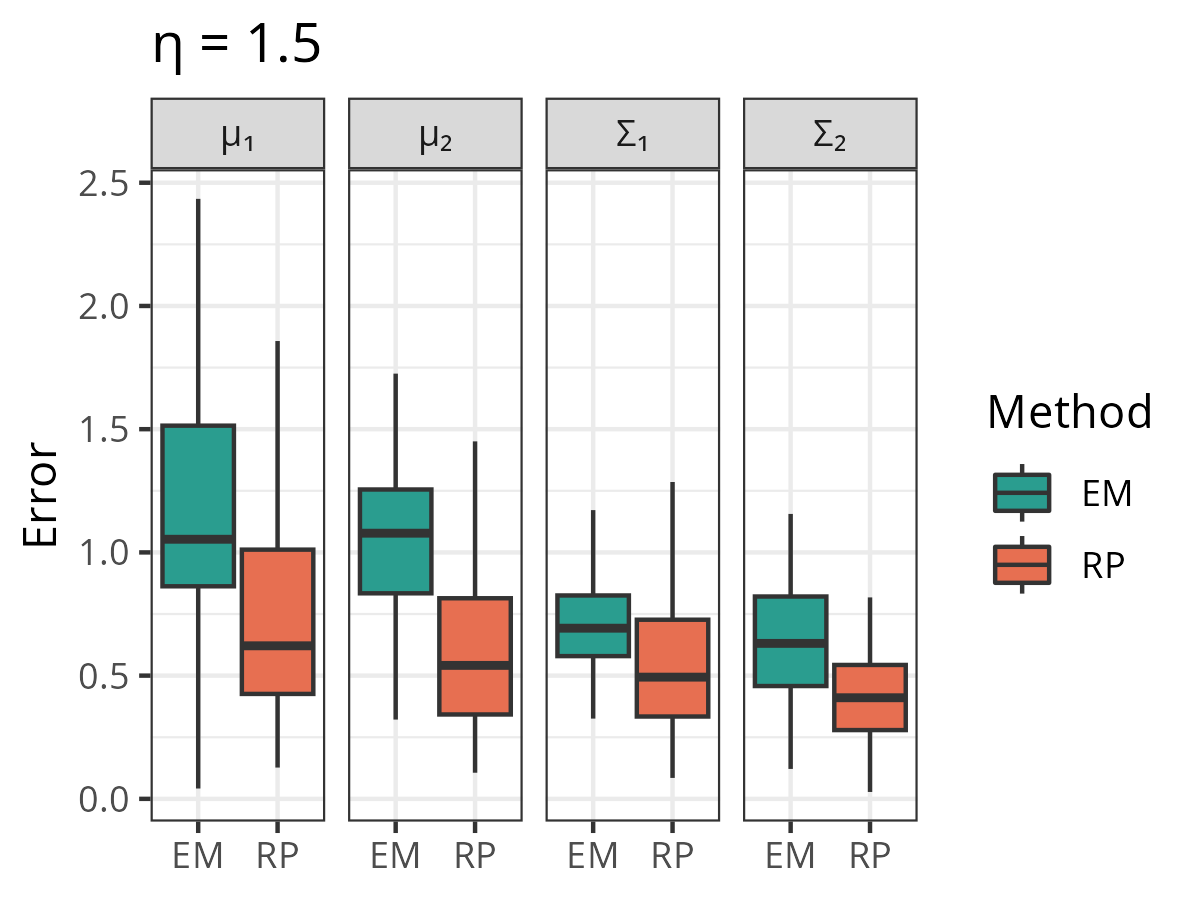}}
		\subfloat{\includegraphics[width=67mm]{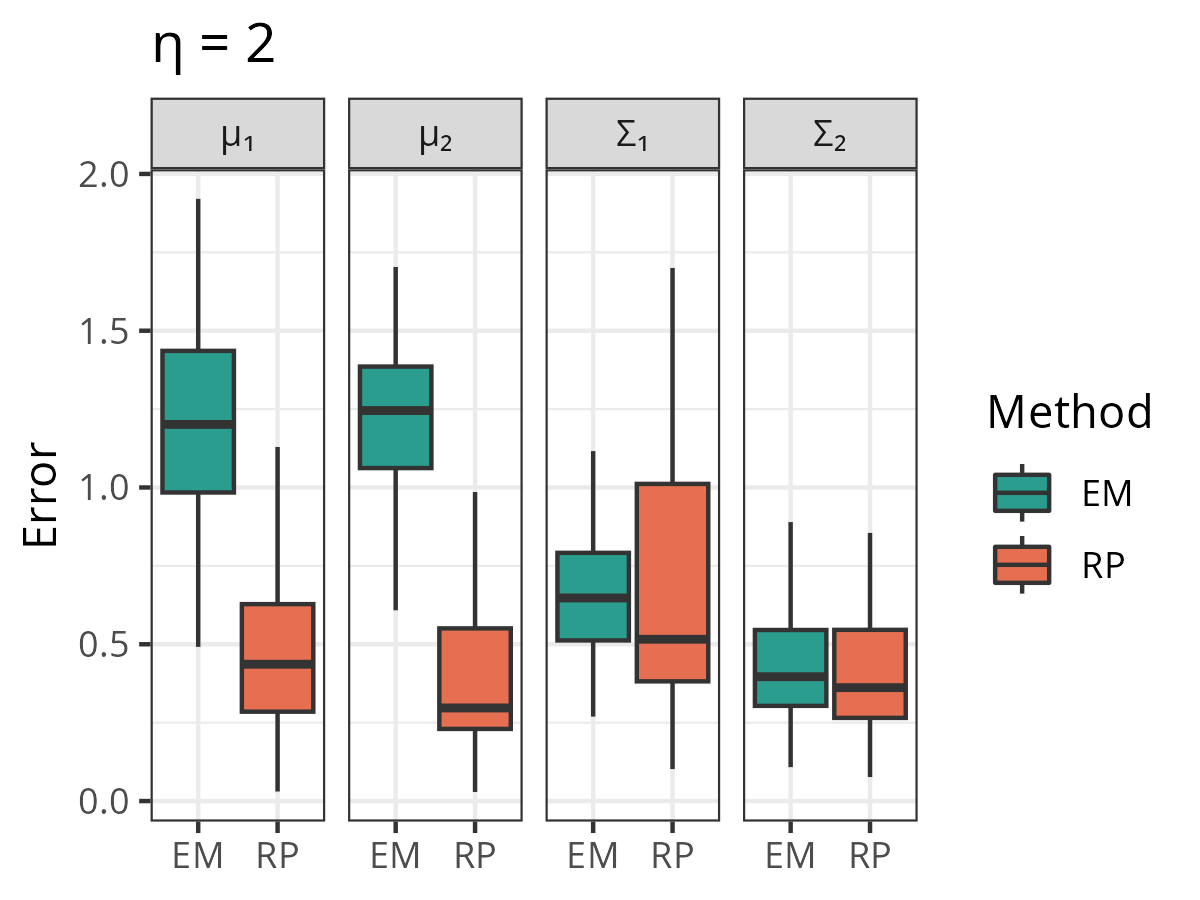}}
\caption{Comparison between the EM-st (EM) and RP methods via boxplots of estimation errors across the four separability scenarios $\eta \in \{1/2,1,3/2,2\}$. Errors are measured using the $L^2$-distance for the mean vectors, and the Frobenius distance for the scatter matrices.}

		\label{F:SS2}
	\end{figure}
	
	
	\subsubsection{Example 2}
	
We repeat the previous simulation setup, but now the observations are generated from a
two-component \emph{bivariate Student $t$}-mixture with $\nu=4$ degrees of freedom,
contaminated by a fraction $\gamma$ of \emph{uniform} noise on an axis-aligned square.
Specifically, for each replicate we generate $N=500$ observations from
\[
F
= (1-\gamma)\Big\{\lambda_1\, t_{\nu}(\mu_{1}, \Sigma_{1})
                 + \lambda_2\, t_{\nu}(\mu_{2}, \Sigma_{2})\Big\}
  + \gamma\,\mathcal U\!\big([0,4]\times[0,4]\big),
\]
where $\lambda_1=0.3$ and $\lambda_2=0.7$, and
\[
\mu_{1}=(0,0),\qquad \mu_{2}=(2,0),
\]
\[
S_{1}=\begin{pmatrix}1&0\\[0.1em]0&1/2\end{pmatrix},\qquad
S_{2}=\begin{pmatrix}1/2&0\\[0.1em]0&1\end{pmatrix}.
\]
The contamination distribution is uniform on the square centered at $(2,2)$ with side length $4$,
i.e., $\mathcal U([0,4]\times[0,4])$. The parameter $\gamma$ controls the proportion of outliers,
and we consider $\gamma\in\{0.05,\,0.10,\,0.15\}$. Figure~\ref{F:mezOut} displays one simulated dataset for each noise scenario considered in our study.  Each panel corresponds to a different contamination level $\gamma\in\{0.05,0.10,0.15\}$ and illustrates a single realization from the bivariate Student $t$-mixture with uniform outliers; hence, the figure provides a visual summary of how the amount of uniform noise increases across scenarios.

	\begin{figure}[tbp]
		\centering
		\subfloat{\includegraphics[width=137mm]{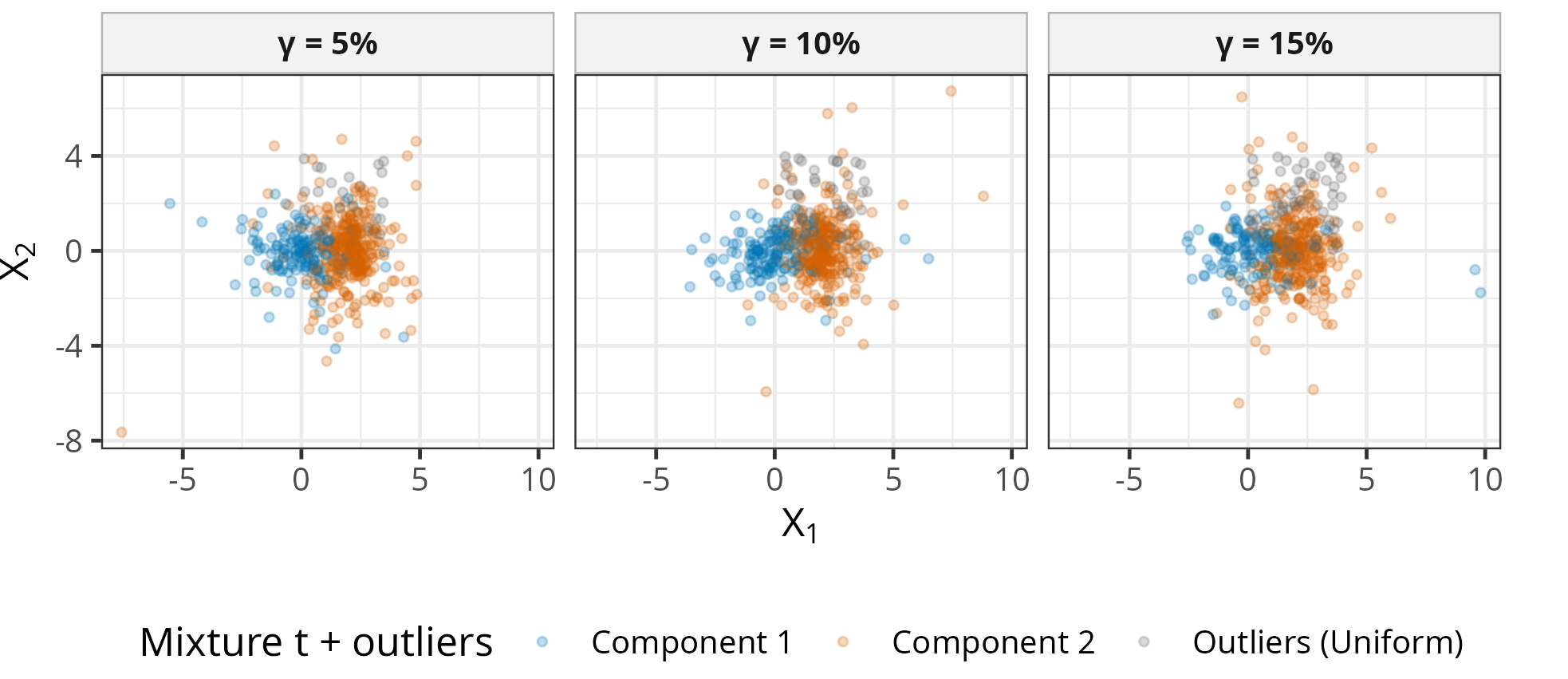}}
		\caption{Simulated bivariate Student $t$-mixture with uniform contamination. Each panel shows a sample of size $N=500$ from a two-component $t_{\eta}$ mixture with $\eta=4$, mixing weights $(\lambda_1,\lambda_2)=(0.3,0.7)$, locations $\mu_1=(0,0)$ and $\mu_2=(2,0)$, and scatter matrices $\Sigma_1=\mathrm{diag}(1,1/2)$ and $\Sigma_2=\mathrm{diag}(1/2,1)$. A proportion $\gamma\in\{0.05,0.10,0.15\}$ of observations is replaced by outliers drawn uniformly from $[0,4]\times[0,4]$.}

		\label{F:mezOut}
	\end{figure}

	In this case we re-estimate the original mixture parameters, and the $L^2$-errors are plotted in Figure \ref{F:SS33}. 
	
	In this framework, our method is compared to a mixture estimation using a robust variant of EM (which we denote by  RobEM). These estimators are introduced in \citep{GR24} and are implemented in the \textit{RGMM} package of the R language. The methodology they propose consists of modifying the EM algorithm, in particular the M-step, where they replace the mean and variance estimates by robust versions derived from the median and the median of the covariance matrix, respectively.
	
In the robust contamination setting, the RP approach exhibits an overall better performance than the robust EM alternative. This improvement is particularly clear for the mixing proportion parameter $\lambda$, whose estimation is more stable under RP across the noise scenarios, as well as for the scatter parameters: the RP reconstructions of the scatter matrices show smaller discrepancies with respect to the true matrices, indicating a higher robustness of RP to the uniform outliers when recovering second-order structure, see Figure \ref{F:SS33} and Figure \ref{F:SS34}.

	\begin{figure}[tbp]
		\centering
		\subfloat{\includegraphics[width=40mm]{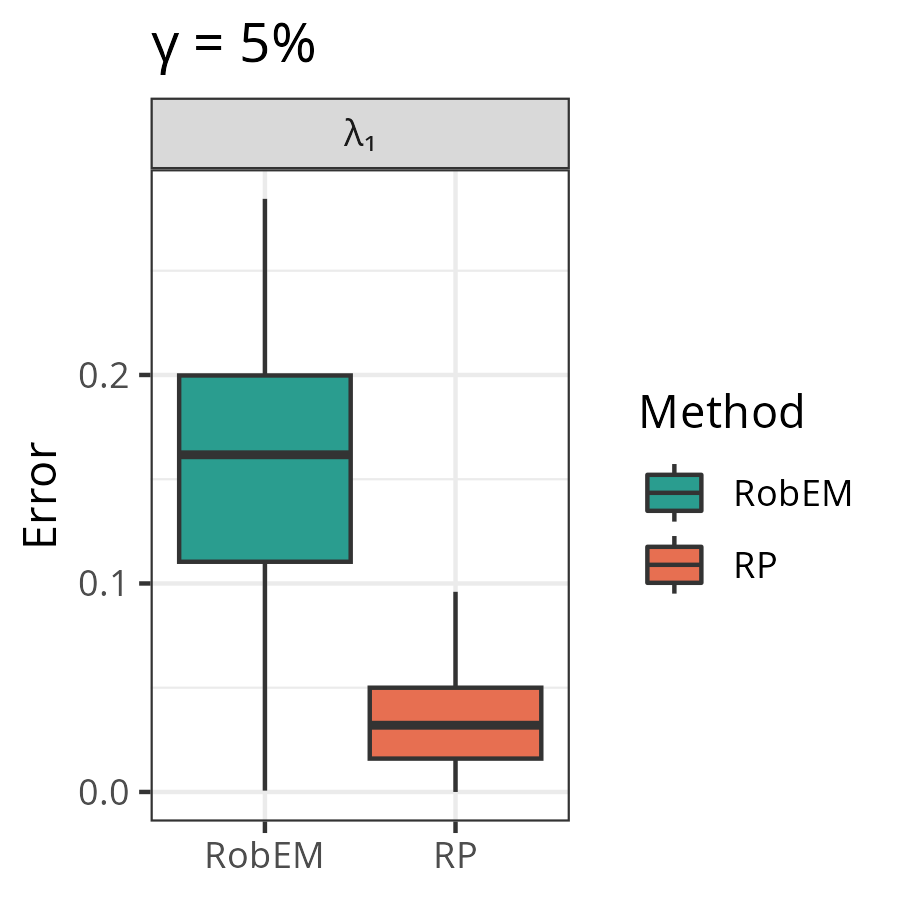}}
		\subfloat{\includegraphics[width=40mm]{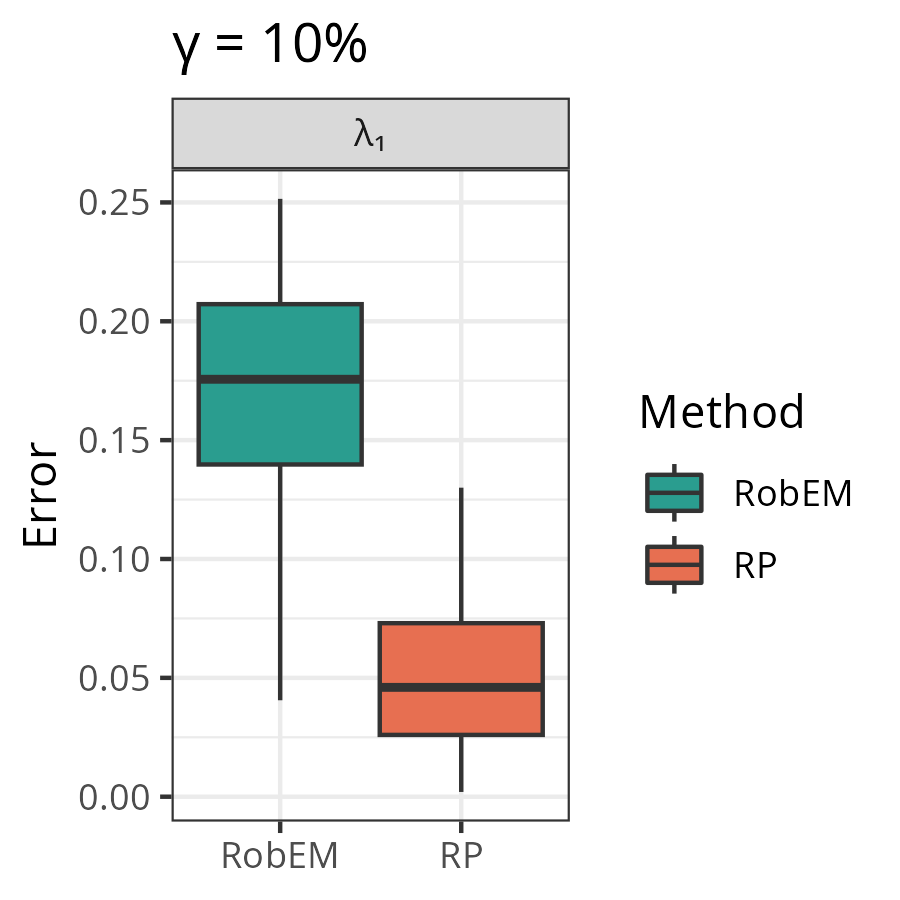}}
		\subfloat{\includegraphics[width=40mm]{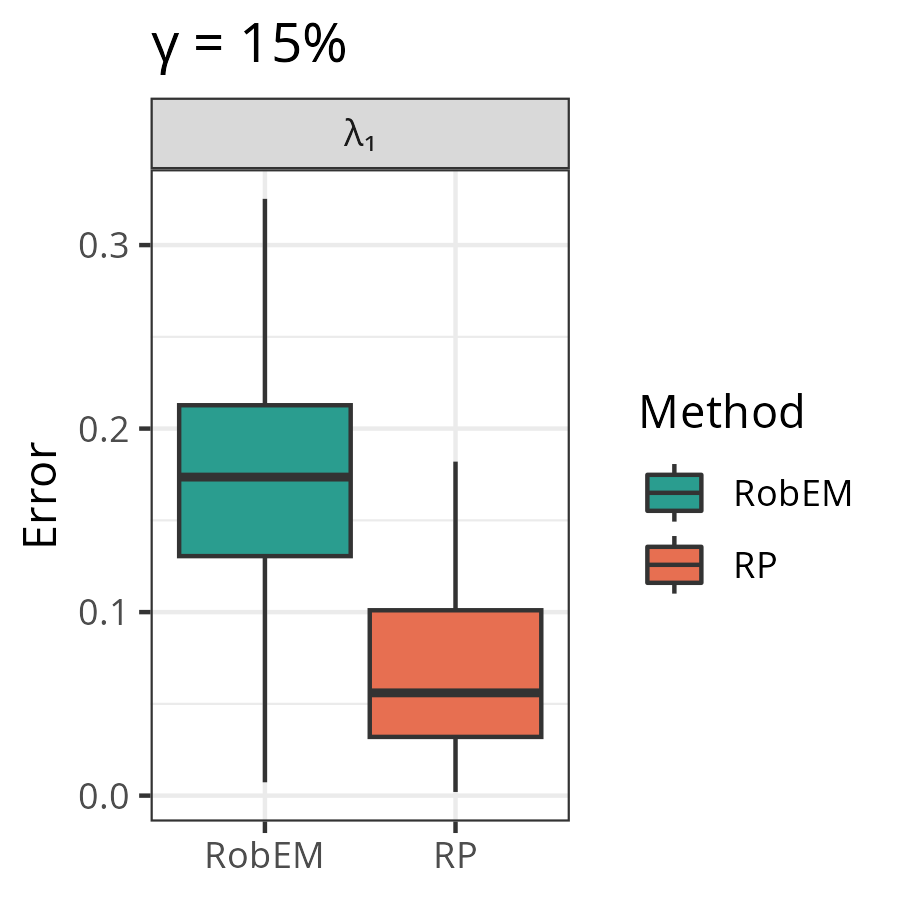}}
\caption{Comparison of estimation errors for the mixing proportion $\lambda_1$ under uniform contamination. Each panel corresponds to a different outlier proportion $\gamma \in \{0.05,0.10,0.15\}$ (left to right), and boxplots summarize the errors over $100$ Monte Carlo replicates for the RobEM and RP methods.}

		\label{F:SS33}
	\end{figure}

	\begin{figure}[tbp]
		\centering
		\subfloat{\includegraphics[width=45mm]{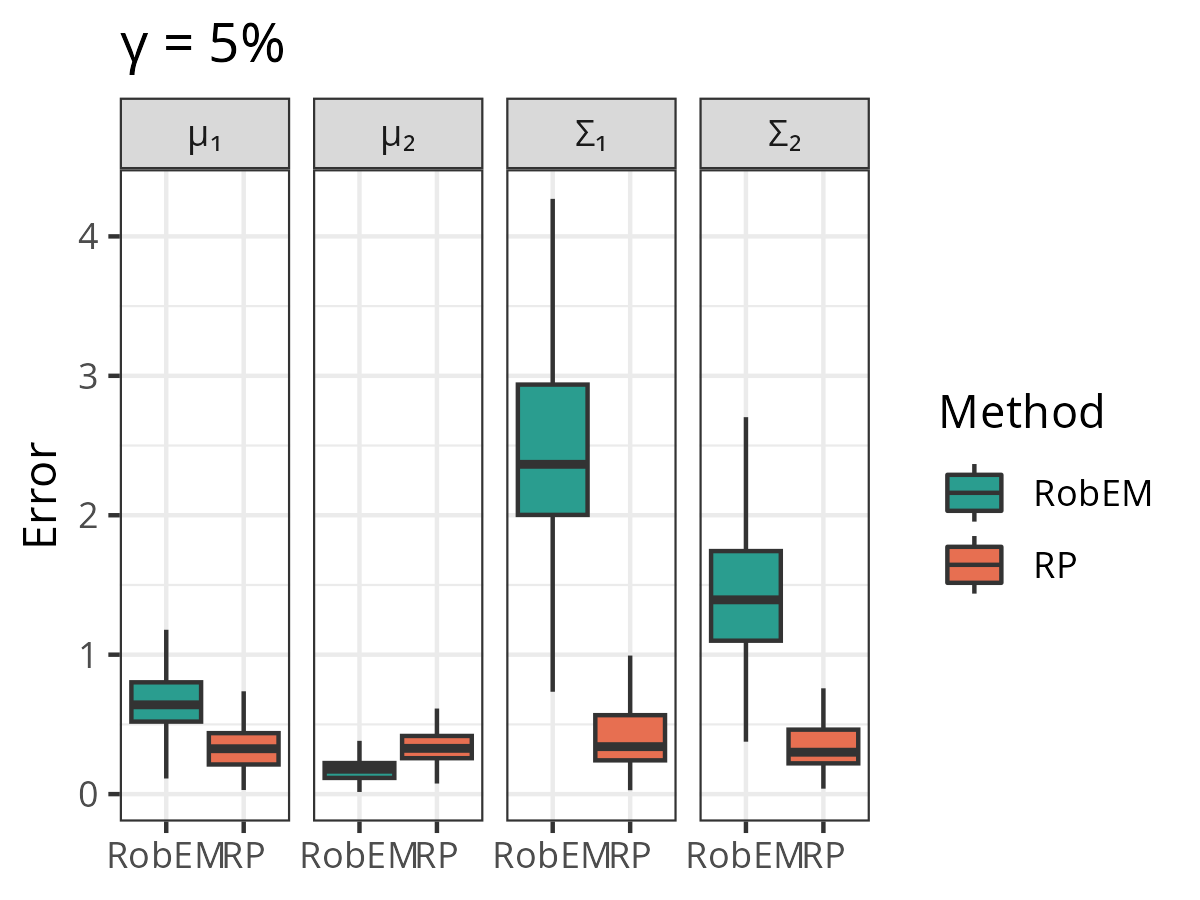}}
		\subfloat{\includegraphics[width=45mm]{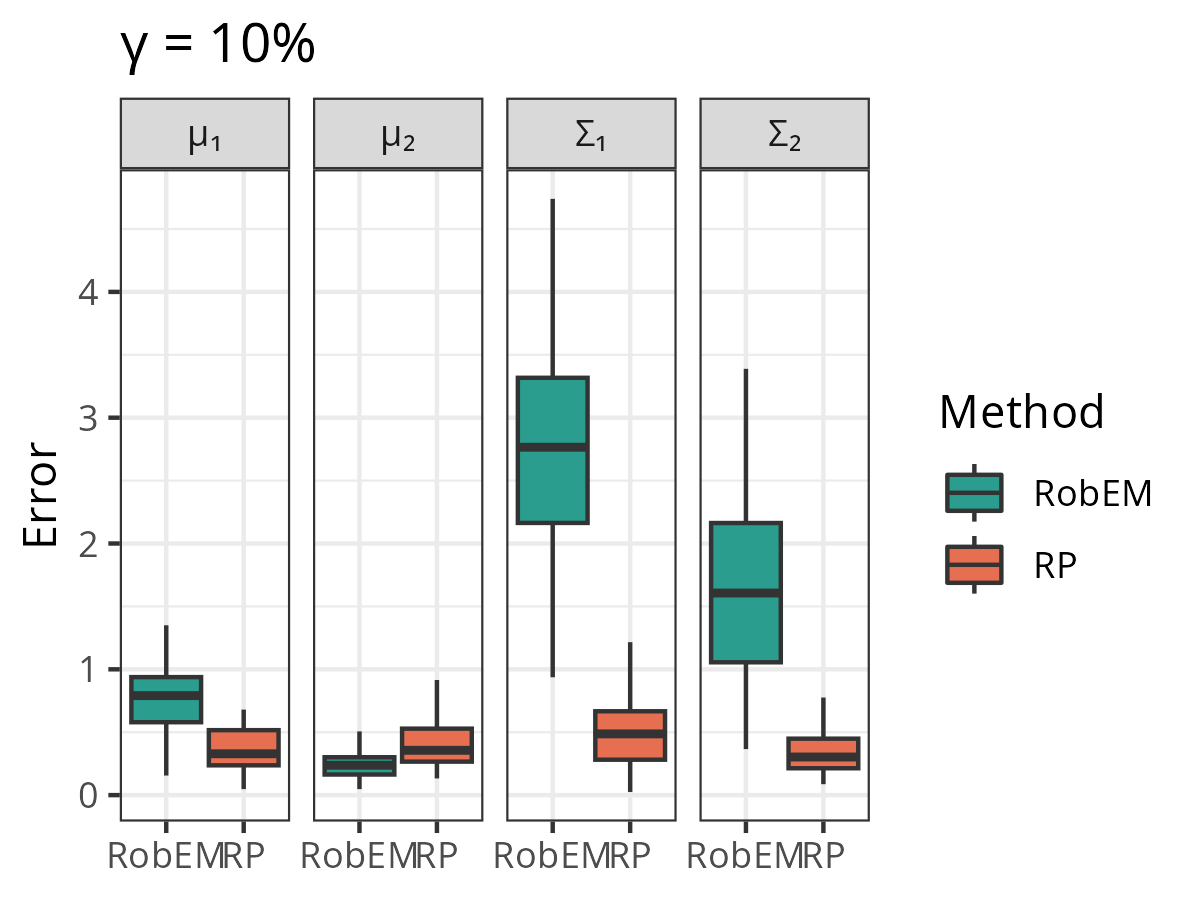}}
		\subfloat{\includegraphics[width=45mm]{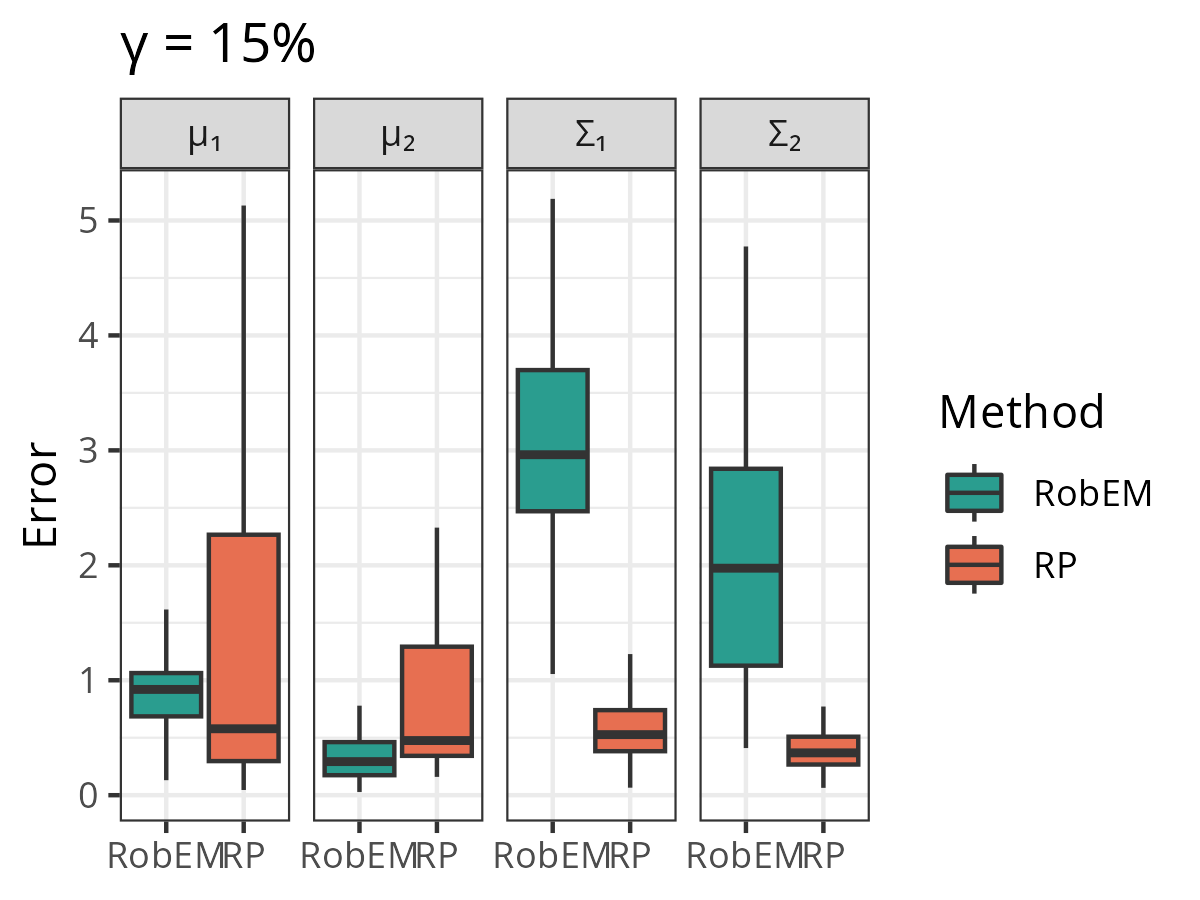}}
\caption{Comparison of estimation errors for the remaining parameters under uniform contamination. Each panel corresponds to a different outlier proportion $\gamma \in \{0.05,0.10,0.15\}$ (left to right), and boxplots summarize the errors over $100$ Monte Carlo replicates for the RobEM and RP methods. Errors are measured using the $L^2$-distance for the mean vectors and the Frobenius distance for the scatter matrices.}

		\label{F:SS34}
	\end{figure}

Table~\ref{Tb:mixla_gamma} reports the average parameter estimates across the replicates, together with their corresponding standard deviations, under the contamination setting with uniform outliers. In addition, we report the (replicate-averaged) confusion matrices (see Table~\ref{Tb:confusion_all2}) computed over the $100$ Monte-Carlo runs---comparing the true component labels with the posterior (MAP) allocations obtained after parameter estimation, for both the robust EM and the RP procedures, see Appendix~\ref{A:confusion}.

	
	\subsection{An example with real data}
	
	The relationship between school performance and students' socioeconomic and cultural status has been widely studied, see \citep{Wh82}.  This association is even stronger in the Latin American countries, see \citep{DBM10}.
	
	In Uruguay, this trend can be seen in the educational evaluation studies of the National Institute for Educational Evaluation (INEEd). A representative sample of $6437$ students in the third year of secondary education in the country  in 2022 is considered. The sample is comprised of students attending two types of educational institutions: Public (free and funded by the state) and Private (depending on the payment of tuition by their students or private sources of financing). In the sample, $4852$ and $1585$ students attend public and private schools, respectively.  By means of a multiple-choice test, using Item Response Theory, a score in Mathematics is assigned to the item. In addition, an index of each student's socioeconomic and cultural level is constructed
	from data collected in a personal questionnaire.\footnote{The database and indexes mentioned above are open and available at  {https://www.ineed.edu.uy/nuestro-trabajo/aristas/}}
	
	Figure~\ref{F:SS334}  depicts  a bivariate histogram in hexagonal cells.  This Figure  shows the positive relationship between the two measures, where an asymmetric distribution is observed. 
	
	\begin{figure}[tbp] 
		\centering
		\subfloat{\includegraphics[width=90mm]{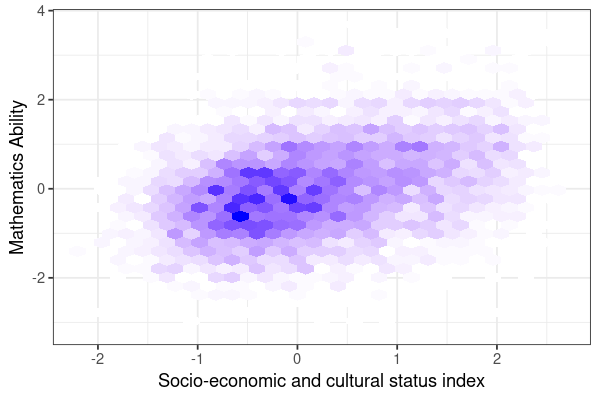}}
		\caption{Hexagonal binning for bivariate data. The plane is tessellated by a regular grid of hexagons (cells). Then the hexagons are plotted using  an intensity of color proportional to the number of points falling in each cell. Higher socio-economic level for smaller values on the $x$-axis, better mathematical ability for larger values on the $y$-axis.}
		\label{F:SS334}
	\end{figure}
	
	The objective is to model this bivariate distribution by means of a mixture of normals and to explain a possible reason for this asymmetry.  Using the method that we have developed, 
	the following estimated parameters of a mixture of two Gaussians are obtained:
	$\hat{\lambda}_1=0.75, \ \hat{\mu}_1=(-0.14,-0.15), \ \hat{\mu}_2=(1.16, 0.47)$,
	and
	\[
	\hat{\Sigma}_{1}= \begin{pmatrix}0.39 & 0.15\\0.15 & 0.65\end{pmatrix},\quad
	\hat{\Sigma}_{2}= \begin{pmatrix}0.16 & 0.08\\0.08 & 0.48\end{pmatrix}.
	\]
	
	These two data patterns are possibly explained by the type of 
	institution that the students attend. In this example, we consider $\lambda$ to be known 
	\textit{a priori}, i.e., the proportion of students attending Public educational institution is $0.75$. 
	If we calculate the vector of means and the matrix of empirical variances and covariances of the sample conditioned to each institution, we obtain
	$\hat{\mu}_{\text{Public}}=(-0.16,-0.180), \ \hat{\mu}_{\text{Private}}=(1.18, 0.51)$,
	and
	\[
	\hat{\Sigma}_{\text{Public}}= \begin{pmatrix}0.45 & 0.10\\0.10 & 0.75\end{pmatrix},\quad
	\hat{\Sigma}_{\text{Private}}= \begin{pmatrix}0.20 & 0.12\\0.12 & 0.59\end{pmatrix}.
	\]	 
	These values are similar to those obtained in the mixture estimation. 

An important aspect of this example is the comparison of student performance in acquired mathematical skills between those attending private schools and those attending public schools. The comparison considers two variables: the score assigned to the item and the socioeconomic and cultural level of the child's home. Comparing the means $\hat \mu_1 =(-0.14, -0.15)$ for those attending public schools with $\hat \mu_2=(1.16, 0.47)$ for those attending private schools, we see a clear drop in the performance of students attending public schools, particularly marked in the first coordinate, but also significant in the second.
This highlights a  problem to be considered by the country's policymakers.

\section{A distributional discrepancy for model-based clusterings}\label{S:randompartitions}

\subsection{Introduction}

The problem of comparing two different partitions of a finite set of observations is a classical topic in the clustering literature. A large number of indices have been proposed for that purpose, including the Rand index,
the adjusted Rand index, the Jaccard index, the Dunn index, the silhouette index, and the Xie--Beni index, among others; see, for instance, \citep{AB73} and  \citep{FM83}.

Our aim in this section is slightly different. We do not introduce a direct metric on partitions. Instead, we compare the fitted mixture distributions associated with two model-based clusterings. Thus, the object of comparison is distributional rather than purely combinatorial. This yields a model-based distributional discrepancy with a natural interpretation in terms of the underlying probability measures.

More precisely, let $P$ and $Q$ denote the fitted mixture distributions associated with two model-based clustering procedures. We compare $P$ and $Q$ through one-dimensional projections and aggregate projection-wise Kolmogorov--Smirnov discrepancies. If the induced mixture distributions are close, then the corresponding clusterings can be regarded as close from a
model-based distributional viewpoint. Throughout this section, the theoretical results are stated under the assumption that both fitted distributions belong to the same mixture class with a common number $m$ of components. Extending the
theory to different numbers of components is beyond the scope of the present paper.

On the other hand, a probabilistic model used quite often in Bayesian statistics as well in cluster analysis 
is to consider that the data are generated by a mixture of multivariate normal distributions. 
Indeed, a popular and well-studied clustering method is to use 
Gaussian mixture models (GMM) for data clustering (model-based clustering). 
One can perform hard clustering or soft clustering. 
For hard clustering the GMM model assigns each datum 
to the component that maximizes the component posterior probability, 
given the datum, and each datum is assigned to only one cluster. 
Soft clustering assigns each instance a probability of belonging to a cluster.
See for instance Gormley et al.\ \citep{GMR23} for details.

From the recent article   \citep{GMR23} we quote:
\textit{``Through its basis in a statistical modeling framework, 
	model-based clustering provides a principled and reproducible approach to clustering. 
	In contrast to heuristic approaches, model-based clustering allows for robust approaches
	to parameter estimation and objective inference on the number of clusters, 
	while providing a clustering solution that accounts for uncertainty in cluster membership \dots''}

Instead of comparing the partitions themselves, we compare the fitted mixture distributions associated with two model-based clustering procedures. More precisely, the objects of comparison are the probability measures $P$ and $Q$, not the label assignments. Accordingly, the quantities $D_k$ and $MA_k$ introduced below should be interpreted as projection-based distributional discrepancy measures between $P$ and $Q$, rather than as direct measures of agreement between partitions.


As pointed in Baudry et al.\ \citep{BRCLG10},  
several problems appear when performing Gaussian-model-based clustering, 
affecting the number of clusters obtained by BIC. Some nice solutions are proposed.  
However these problems do not affect us. An important aspect of our approach is that, once the fitted model-based distributions have been obtained, the proposed discrepancy can be computed directly, and it is readily applicable in high-dimensional settings.

We begin with the two independent-samples case.


\subsection{Description of the algorithm} \label{S:algorithms1}

Suppose that we wish to compare two fitted model-based distributions arising from two clustering procedures. We consider two independent samples 
$\aleph_1=\{X_1,\dots,X_\ell\}\subset\RR^d$ and 
$\aleph_2=\{Y_1,\dots,Y_r\}\subset\RR^d$, 
with distributions $P$ and $Q$ respectively, both being mixtures of $m$ multivariate Gaussian or $t$-distributions.

Our method relies on random projections and Kolmogorov--Smirnov (KS) distances. Specifically, each sample is projected onto several random one-dimensional subspaces, empirical distributions are computed, and the KS distances are evaluated between the projected samples. Aggregating these distances provides a projection-based measure of discrepancy between the two induced mixture distributions.

  
Let $P_\ell$ and $Q_r$ be the empirical distributions of $\aleph_1$ and $\aleph_2$. 

Applying Theorems~\ref{T:gaussmixture} and~\ref{T:tmixture}, the discrepancy between the empirical distributions $P_{\ell}$ and $Q_{r}$, and hence between the underlying fitted distributions $P$ and $Q$, is measured by the average of the projected KS distances along a strong sm-uniqueness set of random directions.


For each direction $u_i$, denote the projected samples by
\[
\mathcal{M}_{1i}=\{\langle u_i,X_j\rangle\}_{j=1}^\ell, 
\qquad 
\mathcal{M}_{2i}=\{\langle u_i,Y_j\rangle\}_{j=1}^r,
\]
with corresponding empirical CDFs $F_{1i}$ and $F_{2i}$.  
The one-dimensional Kolmogorov--Smirnov statistic is
\[
\mathrm{KS}(i)=\sup_{t\in\RR}|F_{1i}(t)-F_{2i}(t)|.
\]

We then define
\[
D_k := \max_{1\le i\le k}\mathrm{KS}(i).
\]

An alternative measure is the average
\[
\mathrm{MA}_k := \frac{1}{k}\sum_{i=1}^k \mathrm{KS}(i),
\]
where $\mathrm{KS}(i)$ is the KS distance between the empirical distributions of $\mathcal{M}_{1i}$ and $\mathcal{M}_{2i}$. 

Theorems~\ref{T:gaussmixture} and~\ref{T:tmixture} show that if this limit is zero, then $P=Q$.  
The asymptotic distributions of $\mathrm{MA}_k$ and $D_k$ are unknown, since projections onto different directions are only conditionally independent.  However, we can use the bootstrap method to proceed in practice.

Alternatively if the sample sizes $\ell$ and $r$ are large, we split them into $k$ disjoint subsamples $\aleph_{split}:=(\aleph_{11}, \aleph_{21}), \ldots ((\aleph_{1k}, \aleph_{2k}))$, and use different subsamples for each direction $u_i$.  
In this case, $\mathrm{MA}_k$ (or $D_k$) becomes an average (or maximum) of $k$ independent one-dimensional KS statistics, which is distribution-free.  
This strategy, however, comes at the cost of reduced statistical power.

The overall scheme is summarized in Algorithm~\ref{A:algo2} for $k$ disjoint subsamples; an analogous scheme can be implemented for the bootstrap method.

\begin{algorithm}[H]
	\caption{Distributional discrepancy via random projections and Kolmogorov--Smirnov distances}
	\label{A:algo2}
	\begin{algorithmic}[1]
		\Require Integers $m,d$; subsamples $\aleph_{split}$ ; fix $k\ge \tfrac{1}{2}(2m-1)(d^2+d-2)+1$
		
		\State Draw $k$ random directions $u_1,\dots,u_k\sim \mathrm{Unif}(\mathbb{S}^{d-1})$
		\For{$i=1$ \textbf{to} $k$}
		\State Project subsamples: $\mathcal{M}_{1i}=\{\langle u_i,\aleph_{1i}\rangle\}$, \quad $\mathcal{M}_{2i}=\{\langle u_i,\aleph_{2i}\rangle\}$
		\State Compute empirical CDFs $F_{1i},F_{2i}$ from $\mathcal{M}_{1i}$ and $\mathcal{M}_{2i}$
		\State $\mathrm{KS}(i)\gets \sup_{t\in\RR}|F_{1i}(t)-F_{2i}(t)|$
		\EndFor
		\State \Return $\mathrm{D}_k \gets \max_{i=1,\dots,k}\mathrm{KS}(i)$ or $\mathrm M_k \gets   \frac{1}{k}\sum_{i=1}^k \mathrm{KS}(i)$
	\end{algorithmic}
\end{algorithm}

\begin{remarks}
\begin{enumerate}

	\item Pre-whitening via $\Sigma^{-1/2}$ is a convenient option to mitigate scale (and unit-of-measure) effects when the practitioner wishes to control for them; in that case, one may apply a whitening transform to the data before computing the projected KS aggregations.

	\item In the one-sample setting with two clustering procedures, we start from a single sample $\aleph=\{X_1,\dots,X_N\}\subset\RR^d$. We split $\aleph$ into two disjoint subsamples $\aleph_1$ and $\aleph_2$, and apply a different clustering procedure to each subsample (e.g., GMMs fitted under different prior specifications). Note that if the data are mapped through a linear transformation (in particular, a one-dimensional projection) while the GMM is fitted in the original space, then the induced distribution on the projection subspace is again a Gaussian mixture, with mean vectors and covariance matrices obtained by applying the same linear map to the fitted parameters. Consequently, for each random direction we may fit a univariate Gaussian mixture to the projected data, using mixture weights given by the cluster proportions, and then compute $D_k$ or $\mathrm{MA}_k$.

	\end{enumerate}
\end{remarks}


  \subsection{Simulations}
  
  Consider two mixtures of two bivariate Gaussians, with parameters
  \begin{align*}
  	F_{1}&:=\lambda_1 N \left(\mu_{1,1}, \Sigma_{1,1} \right)+ (1-\lambda_1) N \left(\mu_{2,1}, \Sigma_{2,1} \right),\\
  	F_{2}&:=\lambda_2 N \left(\mu_{1,2}, \Sigma_{1,2}\right) + (1-\lambda_2) N \left(\mu_{2,2}, \Sigma_{2,2}\right), 
  \end{align*}
   $\mu_{1,1}=\mu_{1,2}:=(1,-1), \mu_{2,1}=(-2,2),   \lambda_1 = 0.5$, and
  $\Sigma_{1,1}=\Sigma_{1,2}:=\begin{pmatrix}1 & 0\\0 & 2\end{pmatrix}$ , $\Sigma_{2,1}:= \begin{pmatrix}3 & 1\\1 & 4\end{pmatrix}$
 and where, given $\eta_1,\eta_2,\eta_3$, we set
  $\mu_{2,2}:=\mu_{2,1} + (\eta_1,\eta_1)$, $\Sigma_{2,2} := (1+\eta_2) \Sigma_{2,1}$, and $\lambda_2:=0.5+ \eta_3$.
  
  The distances between the two mixtures $F_1$ and $F_2$ are calculated for three scenarios:
  \begin{description}
  	\item[Scenario 1:]$\eta_1$ on a grid of values in $[0,1]$ and  $\eta_2=\eta_3=0$, 
  	\item[Scenario 2:]$\eta_2$ on a grid of values in $[0,2]$ and  $\eta_1=\eta_3=0$, 
  	\item[Scenario 3:]$\eta_3$ on a grid of values in $[0,0.5]$ and  $\eta_1=\eta_2=0$.
  \end{description}
  
  \begin{figure}[tbp]
  	\centering
  	\subfloat{\includegraphics[width=65mm]{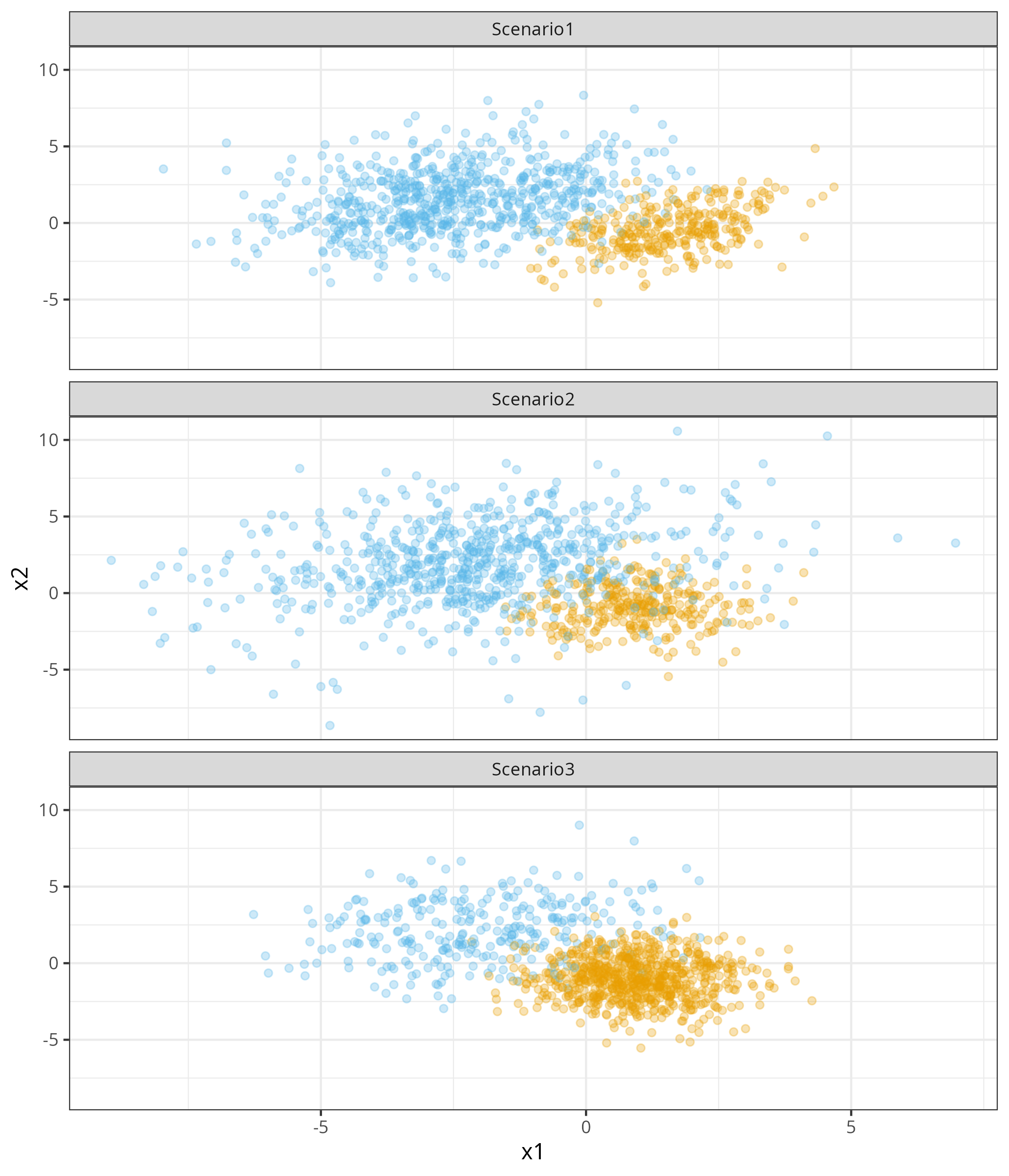}}
\caption{Representative realization in each scenario comparing mixtures $F_1$ and $F_2$: Scenario~1 varies the location shift $\eta_1$ (with $\eta_2=\eta_3=0$), Scenario~2 varies the scale factor $\eta_2$ (with $\eta_1=\eta_3=0$), and Scenario~3 varies the mixing-weight perturbation $\eta_3$ (with $\eta_1=\eta_2=0$).}

  	\label{F:vis}
  \end{figure}
  
\begin{figure}[tbp]
  	\centering
  	\subfloat[Scenario 1]{\includegraphics[width=40mm]{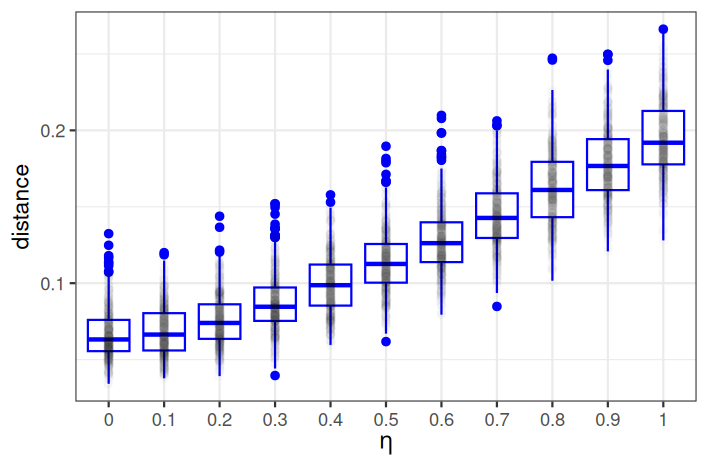}
  	\includegraphics[width=40mm]{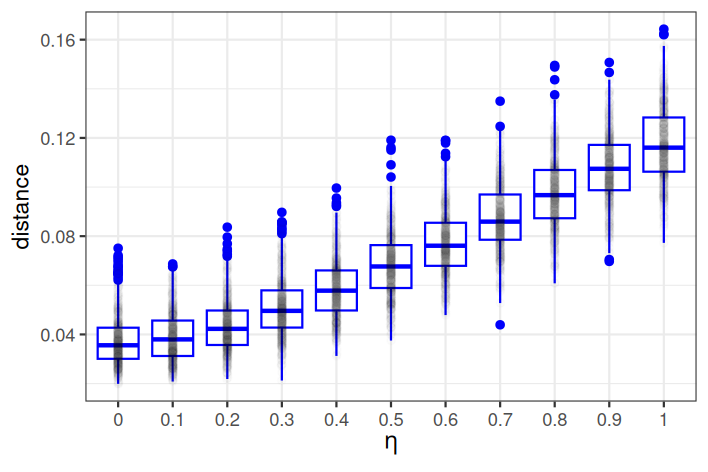}  }
  	
  	\subfloat[Scenario 2]{\includegraphics[width=40mm]{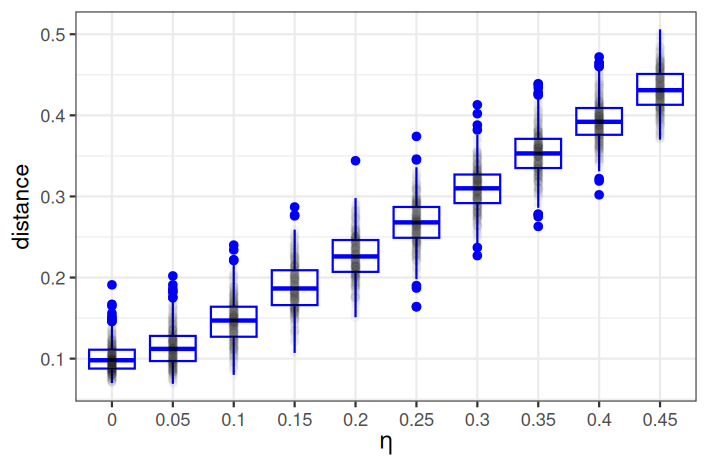}
  	\includegraphics[width=40mm]{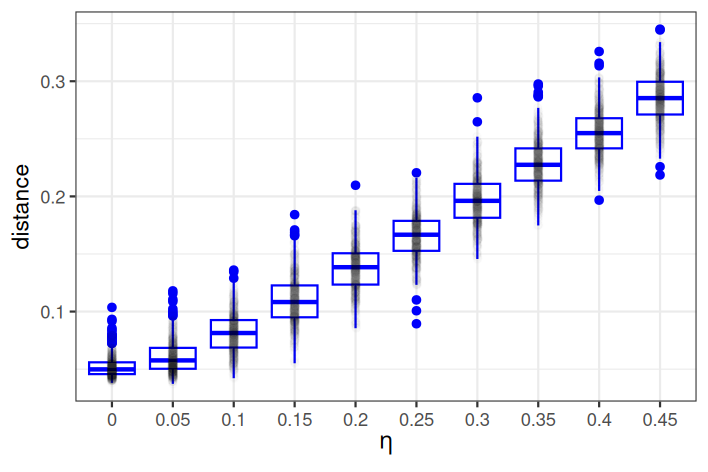} }
  	
  	\subfloat[Scenario 3]{\includegraphics[width=40mm]{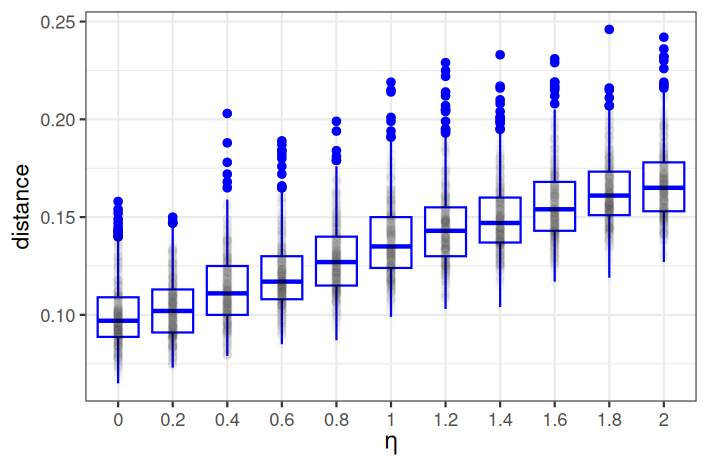}
  	\includegraphics[width=40mm]{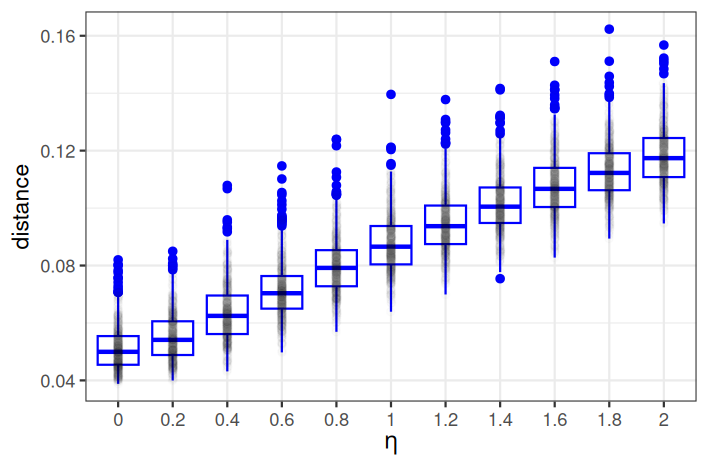} } 
\caption{Agreement between mixtures $F_1$ and $F_2$ using random projections: boxplots of $D_k=\max_{1\le i\le k}\mathrm{KS}(i)$ (left) and $\mathrm{MA}_k=\frac1k\sum_{i=1}^k \mathrm{KS}(i)$ (right), over 500 replicates with $k=100$ directions and 500 draws from each mixture per replicate. Scenario~1 varies location ($\eta_1$), Scenario~2 varies scale ($\eta_2$), and Scenario~3 varies mixing weights ($\eta_3$).}
  	\label{F:orden}
\end{figure}

In Figure~\ref{F:vis}, for improved visualization, one representative simulation per scenario is displayed, corresponding to \(\eta_1 = 1\), \(\eta_2 = 2\), and \(\eta_3 = 1/2\) in Scenarios 1, 2, and 3, respectively. The maximum aggregation measure \(D_k\)  quantifies the largest observed separation between the mixtures \(F_1\) and \(F_2\) across the $k$ projected KS comparisons. This measure is particularly sensitive to pronounced discrepancies (whether in location, dispersion, or component weights), thereby providing a robust indicator of distinct distributional features.

Figure~\ref{F:orden} presents, for each scenario, boxplots of the $500$ replicates of \(D_k\), obtained by simulating 500 observations from each mixture and repeating the procedure 500 times. When \(\eta_1=\eta_2=\eta_3=0\), the mixtures coincide and \(D_k\) attains minimal values. As the mixtures diverge, \(D_k\) exhibits a clear increasing trend, faithfully reflecting the growing separation.

The same analysis was carried out using the  average aggregation measure  \(\textit{MA}_k\), obtaining similar results, as shown in the right column of Figure~\ref{F:orden}.

\section{Conclusions and future work}

The article begins with a review of random-projection methods in statistics.

Then we consider two important statistical problems: that of estimating for mix-
tures of multivariate normal distributions and mixtures of t-distributions, and that of
measuring a discrepancy between mixture distributions induced by two different
model-based clusterings.

A model-based distributional discrepancy is introduced. Rather than introducing a direct metric on partitions, we consider the
distance between the mixture distributions associated with them, which provides a distributional comparison between the corresponding fitted model-based distributions.

 Instead of looking at the number of matchings, we look at the distance between the mixture distributions associated with them,   which provides more information about the discrepancy between the associated mixture distributions.

The relationship between school performance and students' socioeconomic and cultural status is analyzed in a real-data example. We find two different patterns that are possible, explained by the type of institution that the students attend : public or private.

The number of directions $k$ into which we project is not a smoothing parameter that should be chosen in an optimal way. In practice we suggest to use a larger value for $k$ (without increasing too much the computational time) than the one provided by Theorem~\ref{T:gaussmixture}, which corresponds to the case when the mixtures are known, whereas we are dealing with empirical distributions which are close to the unknown underlying distributions but not exactly equal. Depending on the problem, among the directions chosen at random it may appear some directions which are close to bad ones. Increasing the value of $k$ solves the problem. In our experience after a reasonable value the performance stabilizes. 

The two parts of the manuscript are linked by a common projection-based representation for Gaussian/$t$-mixtures: the same finite family of directions that supports identifiability and reconstruction in Section~\ref{S:GMestimator} also yields a natural distributional notion of partition agreement in Section~\ref{S:randompartitions}, obtained by aggregating one-dimensional discrepancies across directions.

Some more computationally oriented work, which includes different mixtures of distributions, sample sizes and number of clusters, has been   omitted from the present manuscript.
Also, rates of convergence and the asymptotic distribution of the estimators remain as interesting open problems to be addressed in future work.


\section*{Acknowledgments}
We are grateful to the four anonymous referees,  for their thoughtful suggestions and detailed feedback, which led to substantial improvements in the presentation and overall quality of the manuscript.

Fraiman and Moreno were supported by grant FCE-3-2022-1-172289 from ANII (Uruguay), the Program for the Development of Basic Sciences--PEDECIBA (Uruguay), and 22520220100031UD from CSIC (Uruguay). Ransford was supported by NSERC Discovery Grant RGPIN-2020-04263.

\begin{appendices}
\begin{appendixcompact}

\section{Algorithm~1: Methodological aspects }
\label{apend1}

This appendix details the implementation of Algorithm~\ref{alg:twostep}.

\subsection{Step 1: Estimation of the mixing weights via one-dimensional ECF fits}
\label{apend1:step1}

Let $P$ be an $m$-component Gaussian mixture on $\RR^d$ with parameters
$\Theta=(\Lambda,\bm{\mu},\bm{\Sigma})$, where
$\Lambda=(\lambda_1,\dots,\lambda_m)$, $\sum_{j=1}^m\lambda_j=1$,
$\bm{\mu}=(\mu_1,\dots,\mu_m)$, and $\bm{\Sigma}=(\Sigma_1,\dots,\Sigma_m)$.

Fix a unit direction $u\in\mathbb{S}^{d-1}$ and define the projected data $Y_i(u):=\langle u,X_i\rangle$ with $i=1,\dots,N.$

Under the model, $Y_i(u)$ follows a univariate Gaussian mixture with
component-specific projected means and variances
\[
\mu^{\mathrm{proj}}_{u,j} := \langle u,\mu_j\rangle,
\qquad
\tau^2_{u,j}:=\langle u,\Sigma_j u\rangle,
\qquad j=1,\dots,m,
\]
that is, $P_{\langle u\rangle}
= \sum_{j=1}^m \lambda_j\,\mathcal{N}\!\big({\mu^{\mathrm{proj}}_{u,j}},\tau^2_{u,j}\big).$

\paragraph{Characteristic functions.}
The (model) characteristic function of the projected mixture is
\[
\psi_u(t;\mu^{\mathrm{proj}}_u,\tau_u^2,\Lambda)
:=\sum_{j=1}^m \lambda_j
\exp\!\left(i t\,\mu^{\mathrm{proj}}_{u,j} - \tfrac{1}{2}t^2\,\tau^2_{u,j}\right),
\]
where $\mu^{\mathrm{proj}}_u = (\mu^{\mathrm{proj}}_{u,1},\ldots,\mu^{\mathrm{proj}}_{u,m})$ and
$\tau_u^2=(\tau^2_{u,1},\dots,\tau^2_{u,m})$.
Its empirical counterpart (ECF) is $\hat\psi_{u,N}(t)
:=\frac{1}{N}\sum_{i=1}^N \exp\!\big(i t\,Y_i(u)\big)
=\frac{1}{N}\sum_{i=1}^N \exp\!\big(i t\,\langle u,X_i\rangle\big).$

\paragraph{Grid and moment vector.}
Fix grid points $t_1,\dots,t_M$ (e.g., $t_\ell=\tau\,\ell$ as in \citep{FM1981b}).
Define the empirical moment vector
\[
\hat Z_u
:=\big(\Re \hat\psi_{u,N}(t_1),\dots,\Re \hat\psi_{u,N}(t_M),\
      \Im \hat\psi_{u,N}(t_1),\dots,\Im \hat\psi_{u,N}(t_M)\big)^\top,
\]
and the model moment vector
\begin{align*}
Z_u(\mu^{\mathrm{proj}}_u,\tau_u^2,\Lambda)\\
&:=\big(\Re \psi_u(t_1;\mu^{\mathrm{proj}}_u,\tau_u^2,\Lambda),\dots,\Re \psi_u(t_M;\cdot),\
      \Im \psi_u(t_1;\cdot),\dots,\Im \psi_u(t_M;\cdot)\big)^\top.
\end{align*}

\paragraph{Weighted least-squares / GMM criterion.}
For a positive semidefinite weighting matrix $W\in\RR^{2M\times 2M}$,
we estimate $(\mu^{\mathrm{proj}}_u,\tau_u^2,\Lambda)$ by minimizing
\begin{equation}
\label{E:Qu}
Q_u(\mu^{\mathrm{proj}}_u,\tau_u^2,\Lambda)
:=\big(\hat Z_u - Z_u(\mu^{\mathrm{proj}}_u,\tau_u^2,\Lambda)\big)^\top
   W
   \big(\hat Z_u - Z_u(\mu^{\mathrm{proj}}_u,\tau_u^2,\Lambda) \big),
\end{equation}
over $\Lambda$ in the simplex and $\tau^2_{u,j}>0$.
Following \citep{Tr98,XK10}, $W$ may be estimated using HAC-type procedures
(e.g.\ \citep{NW86}), yielding a generalized method-of-moments estimator
in the sense of \citep{Ha82}.

\paragraph{Multiple directions and label alignment.}
Draw $u_1,\dots,u_k\sim\mathrm{Unif}(\mathbb{S}^{d-1})$ independently and compute,
for each $r=1,\dots,k$, the one-dimensional estimate $({\widehat{\mu}^{\mathrm{proj}}_{u_r,(1)}},\hat\tau ^2_{u_r,(1)},\hat\Lambda_{u_r,(1)}).$

Since each projected fit is performed independently, component labels may be
permuted across directions. We align labels as follows.

Choose a pivot direction $u_{r^\ast}$ (e.g.\ the one maximizing the separation
between the estimated projected means) and order its components by increasing
$\widehat{\mu}^{\mathrm{proj}}_{u_{r^\ast},(1),j}$.
For each $r$, match the $m$ triples
$(\hat\lambda_{u_r,(1),j},{\widehat{\mu}^{\mathrm{proj}}_{u_r,(1),j}},\hat\tau^2_{u_r,(1),j})$
to the pivot triples by solving a minimum-cost assignment problem
(e.g.\ Hungarian algorithm), producing aligned estimates
$(\widetilde{\mu}^{\mathrm{proj}}_{u_r,(1)},\widetilde{\tau}^2_{u_r,(1)},\widetilde{\Lambda}_{u_r,(1)})$.
Finally, define the Step~1 weight estimator by averaging, $\hat\Lambda_{(1)}:=\frac{1}{k}\sum_{r=1}^k \tilde\Lambda_{u_r,(1)}.$

For $t$-mixtures, the projected model remains a univariate $t$-mixture with
projected locations ${\mu^{\mathrm{proj}}_{u,j}}=\langle u,\mu_j\rangle$ and projected scales
determined by $\langle u,\Sigma_j u\rangle$; Step~1 is identical, replacing
$\psi_u(\cdot)$ by the corresponding $t$-mixture characteristic function.

\subsection{Step 2: Re-estimation of projected means/variances and reconstruction in $\RR^d$}
\label{apend1:step2}

With $\Lambda$ fixed at $\hat\Lambda_{(1)}$, we re-estimate, for each direction
$u_r$, the projected location and variance parameters by minimizing the same
criterion \eqref{E:Qu} over $(\mu^{\mathrm{proj}}_u,\tau_u^2)$ only. Denote the resulting
aligned estimators by
\[
{\widehat{\mu}^{\mathrm{proj}}_{u_r,(2)}}=(\widehat{\mu}^{\mathrm{proj}}_{u_r,(2),1},\dots,\widehat{\mu}^{\mathrm{proj}}_{u_r,(2),m}),
\qquad
\hat\tau^2_{u_r,(2)}=(\widehat{\tau}^2_{u_r,(2),1},\dots,\widehat{\tau}^2_{u_r,(2),m}).
\]

\paragraph{Reconstruction of mean vectors.}
For each component $j=1,\dots,m$, we reconstruct $\mu_j\in\RR^d$ via least squares:
\begin{equation}
\label{E:rec_mu_app}
\hat\mu_j
:=\argmin_{\mu\in\RR^d}\ \sum_{r=1}^k\big(\langle u_r,\mu\rangle-\widehat{\mu}^{\mathrm{proj}}_{u_r,(2),j}\big)^2.
\end{equation}
Let $U\in\RR^{k\times d}$ be the matrix with $r$-th row $u_r^\top$ and let
$\hat v_j=(\widehat{\mu}^{\mathrm{proj}}_{u_1,(2),j},\dots,\widehat{\mu}^{\mathrm{proj}}_{u_k,(2),j})^\top$.
When $U^\top U$ is invertible (which holds with probability one for i.i.d.\ random
directions as soon as $k\ge d$), the solution is $\hat\mu_j=(U^\top U)^{-1}U^\top \hat v_j.$

\paragraph{Reconstruction of covariance matrices.}
For each component $j=1,\dots,m$, we reconstruct $\Sigma_j\in$
by solving the convex quadratic program
\begin{equation}
\label{E:rec_Sigma_app}
\hat\Sigma_j
:=\argmin_{\Sigma\in\sym_d^+(\mathbb{R})}\
\sum_{r=1}^k\big(\langle u_r,\Sigma u_r\rangle-\hat\tau^2_{u_r,(2),j}\big)^2.
\end{equation}
As discussed in Proposition~(Uniqueness of covariance reconstruction) in the main text,
uniqueness holds if $\{u_r u_r^\top\}_{r=1}^k$ spans $\sym_d(\mathbb{R})$
(equivalently, if the associated design matrix has rank $d(d+1)/2$).
In practice, \eqref{E:rec_Sigma_app} is solved with a semidefinite-programming
routine (e.g.\ primal--dual interior-point methods, cf.  \citep{BV2004}).

\paragraph{Output of Algorithm~\ref{alg:twostep}.}
The final estimator is $\hat\Theta
=\big(\hat\Lambda_{(1)},\hat{\bm{\mu}},\hat{\bm{\Sigma}}\big)$,   $\hat{\bm{\mu}}=(\hat\mu_1,\dots,\hat\mu_m)$, and 
$\hat{\bm{\Sigma}}=(\hat\Sigma_1,\dots,\hat\Sigma_m).$

\section{Example 3 (Section 3): Constrained-parameter estimation for a $d=20$ three-component Student $t$-mixture}
\label{apend3}

We now consider a high-dimensional Student $t$-mixture in dimension $d=20$, and we compare the estimation obtained by the random projections approach (RP) with a multivariate EM procedure under a \emph{restricted} parameterization. As in the previous examples, we generate i.i.d.\ observations from a finite mixture of multivariate $t$ distributions, but here we impose structural constraints on both the component locations and the scatter matrix in order to reduce the dimension of the parameter space.

Specifically, let $t_{\nu}(\mu,\Sigma)$ denote the $d$-variate Student $t$-distribution with $\nu$ degrees of freedom, location parameter $\mu\in\mathbb{R}^d$, and (scatter) matrix $\Sigma\in\mathbb{R}^{d\times d}$. We consider the three-component mixture
\[
F
:= \lambda_1\, t_{\nu}(\mu_{1}, \Sigma)
 + \lambda_2\, t_{\nu}(\mu_{2}, \Sigma)
 + \lambda_3\, t_{\nu}(\mu_{3}, \Sigma),
\qquad \lambda_1+\lambda_2+\lambda_3=1,
\]
with common degrees of freedom $\nu=2$ and mixing weights $(\lambda_1,\lambda_2,\lambda_3) = (0.3,0.3,0.4).$

\paragraph{Constrained locations.}
To reduce the location-parameter space, we constrain each mean vector to have constant coordinates $\mu_k = m_k\,\mathbf{1}_d,\qquad k=1,2,3,$ where $\mathbf{1}_d$ denotes the $d$-dimensional vector of ones. In the simulation we set $m_1=0,\qquad m_2=1,\qquad m_3=3,$ so that $\mu_1=(0,\ldots,0)$, $\mu_2=(1,\ldots,1)$, and $\mu_3=(3,\ldots,3)$.

\paragraph{Compound-symmetry covariance.}
To reduce the scatter-parameter space, the three components share a common compound-symmetry (CS) matrix, $\Sigma=\Sigma(x):=(1-x)I_d + x\,\mathbf{1}_d\mathbf{1}_d^{\top}$, that is, $\Sigma_{ii}=1$ and $\Sigma_{ij}=x$ for $i\neq j$. In the simulation we take $x=0.25$ (which must satisfy $-1/(d-1)<x<1$ to ensure positive definiteness).

\paragraph{Sampling scheme.}
For each replicate, we generate $N=200$ i.i.d.\ observations from $F$ by first sampling latent labels $Z_i\in\{1,2,3\}$ with
$\mathbb{P}(Z_i=k)=\lambda_k$, and then drawing $X_i \mid (Z_i=k) \sim t_{\nu}(\mu_k,\Sigma),\qquad i=1,\ldots,n$. This experiment is replicated $100$ times.

\paragraph{Estimation and evaluation.}
In each replicate we estimate the mixture parameters under the above constraints using:
(i) a multivariate EM algorithm that estimates $(\lambda_1,\lambda_2,\lambda_3)$, $(m_1,m_2,m_3)$ and $x$; and
(ii) an RP-based procedure that combines (a) a screened collection of random projections, (b) one-dimensional $t$-mixture fits to recover $(m_1,m_2,m_3)$ and $x$, and (c) a short multivariate refinement initialized at the RP estimates.

For both procedures, we compute MAP allocations and summarize clustering performance via the adjusted Rand index (ARI), see Figure \ref{F:ari}. In addition, we measure estimation accuracy through absolute errors for the mixing weights, the mean scalars (reported as errors for $\mu_1,\mu_2,\mu_3$ under the constraint $\mu_k=m_k\mathbf{1}_d$), see Figure \ref{F:g123}; and the scatter matrix, see Figure \ref{F:sigma}.

\begin{figure}[tbp]
  \centering

  \subfloat[$|\mu_1-\mu_1^{*}|$]{\includegraphics[width=30mm]{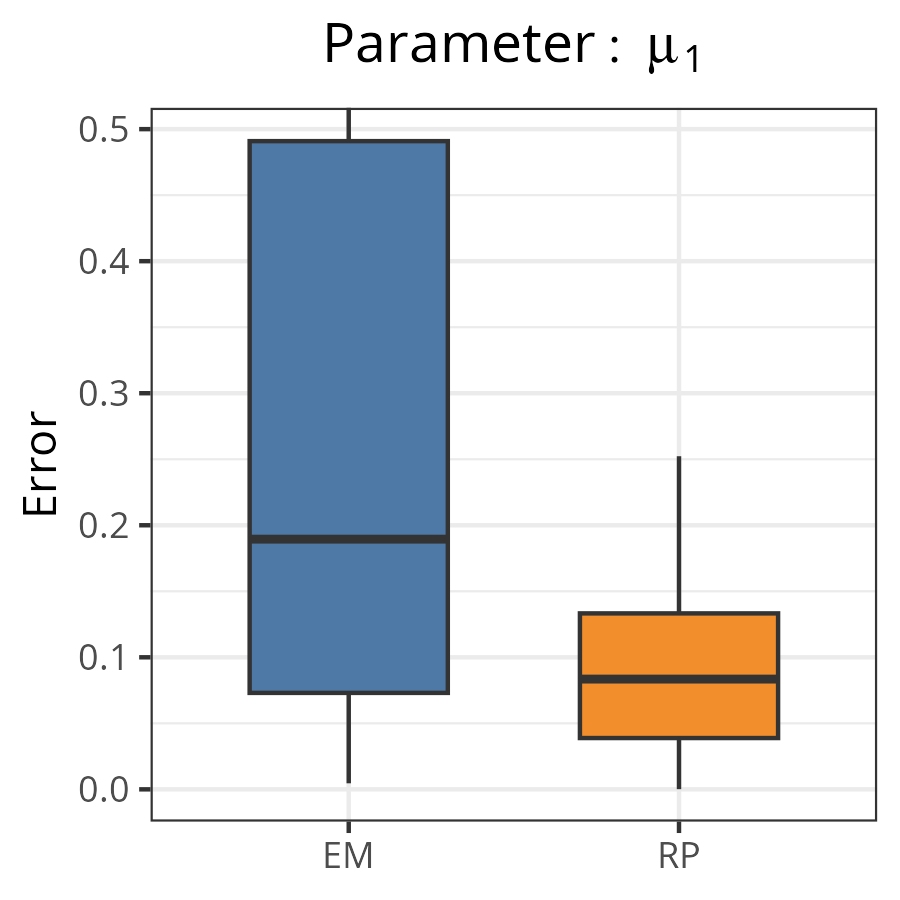}}\hspace{2mm}
  \subfloat[$|\lambda_1-\lambda_1^{*}|$]{\includegraphics[width=30mm]{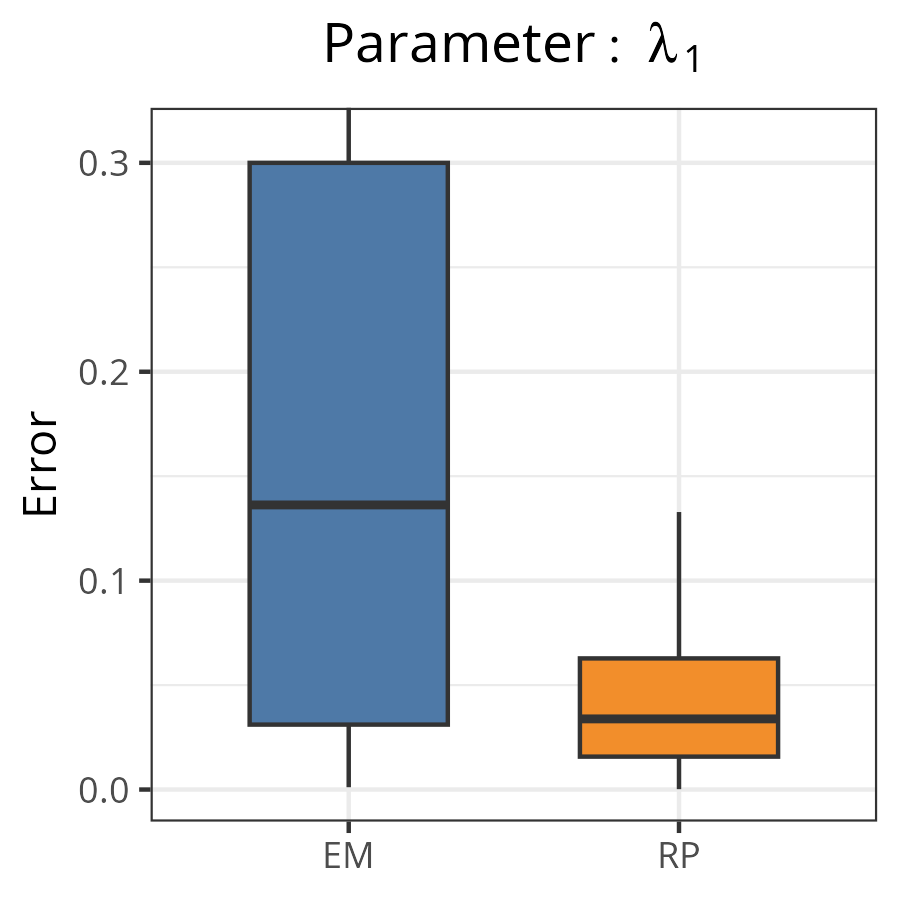}}\\[2mm]

  \subfloat[$|\mu_2-\mu_2^{*}|$]{\includegraphics[width=30mm]{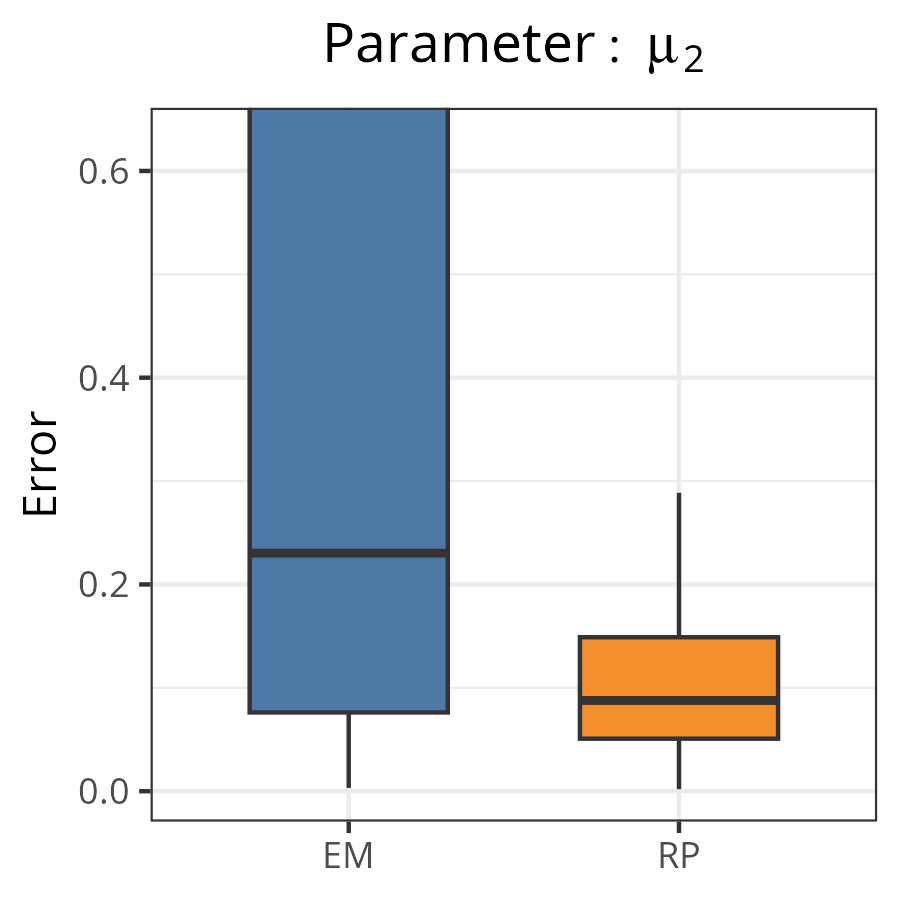}}\hspace{2mm}
  \subfloat[$|\lambda_2-\lambda_2^{*}|$]{\includegraphics[width=30mm]{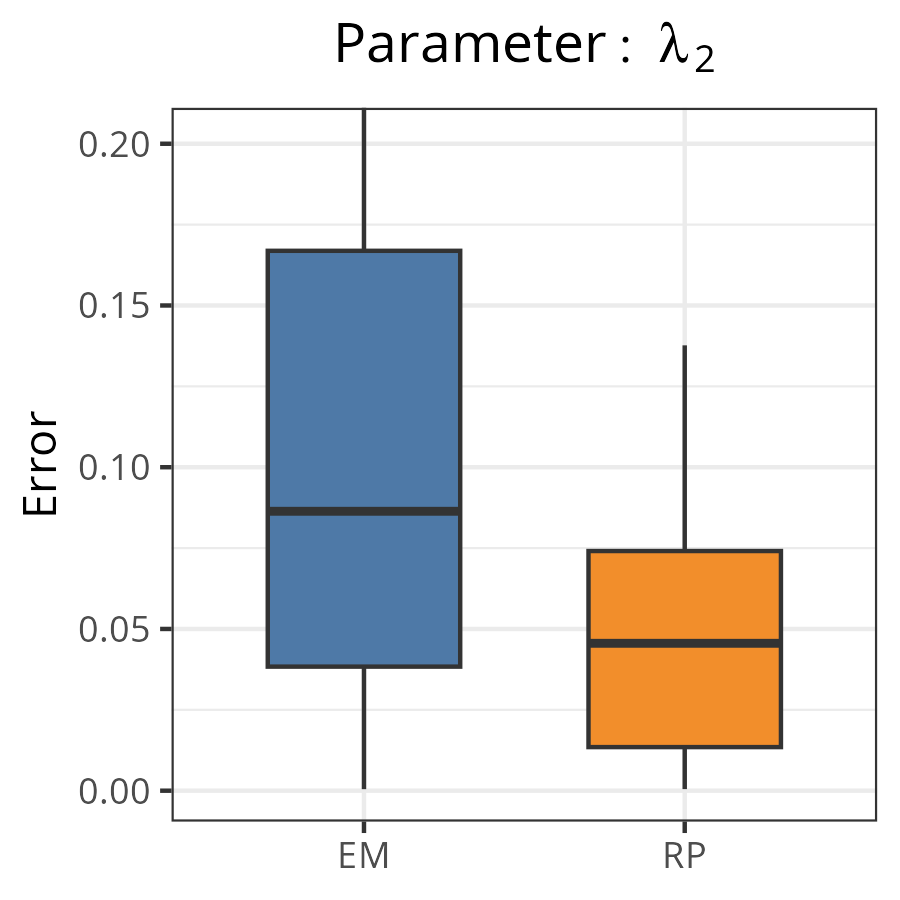}}\\[2mm]

  \subfloat[$|\mu_3-\mu_3^{*}|$]{\includegraphics[width=30mm]{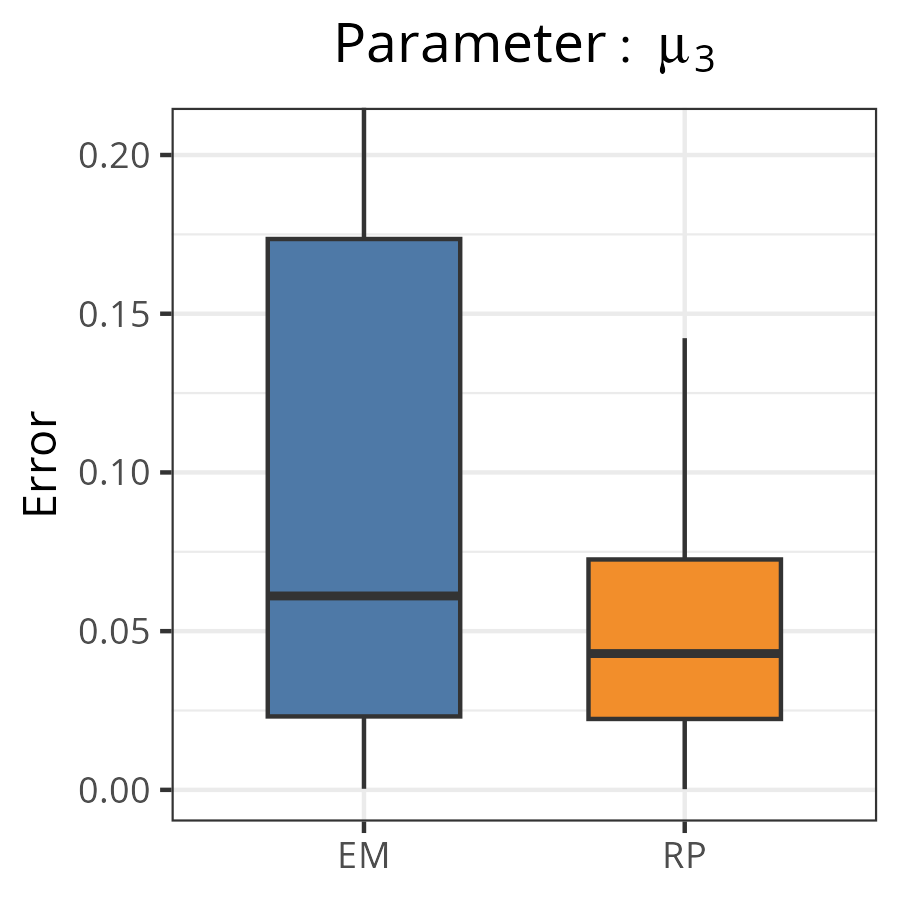}}\hspace{2mm}
  \subfloat[$|\lambda_3-\lambda_3^{*}|$]{\includegraphics[width=30mm]{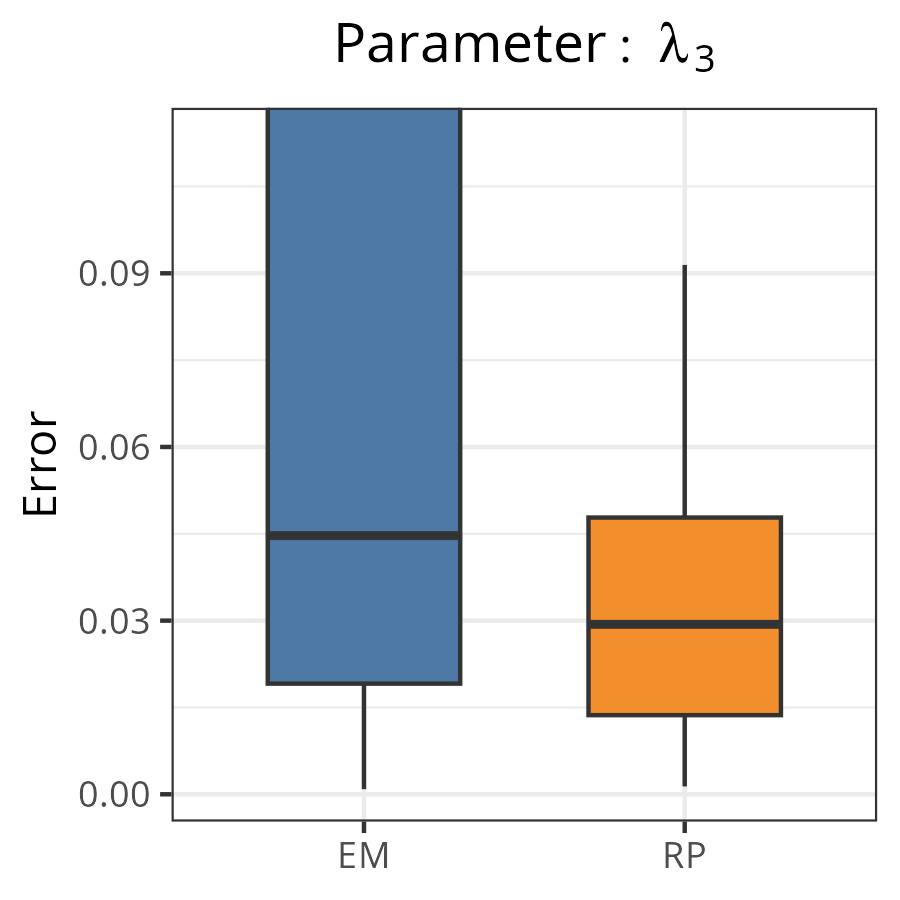}}

  \caption{Estimation errors over 100 Monte Carlo replicates for components $k=1,2,3$: absolute errors for the (constrained) locations $\mu_k$ (left column) and mixing weights $\lambda_k$ (right column), comparing EM vs.\ RP. Boxplots show median and IQR; whiskers follow Tukey's rule (outliers omitted).}
  \label{F:g123}
\end{figure}

\begin{figure}[tbp]
\centering

\begin{minipage}[t]{0.49\textwidth}
  \centering
  \includegraphics[width=60mm]{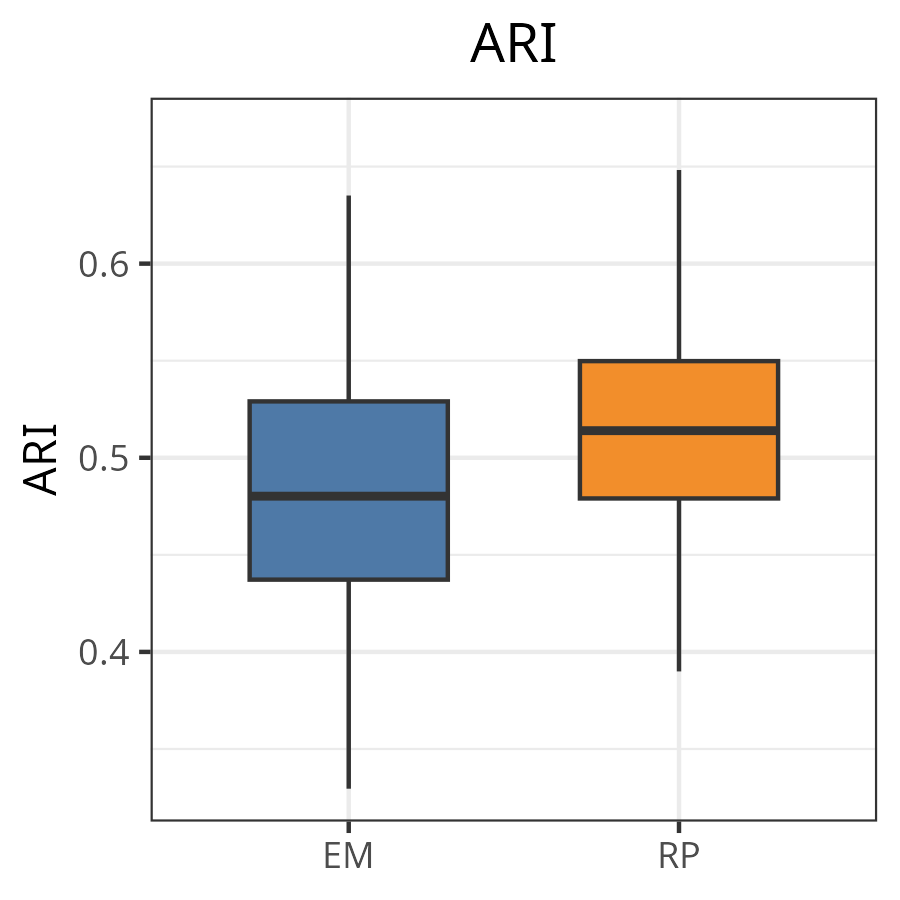}
  \caption{Boxplots of the Adjusted Rand Index (ARI) over the $100$ Monte Carlo replicates, comparing the EM and RP procedures. Larger ARI values indicate a better agreement between the estimated allocations and the true component labels.}
  \label{F:ari}
\end{minipage}\hfill
\begin{minipage}[t]{0.49\textwidth}
  \centering
  \includegraphics[width=60mm]{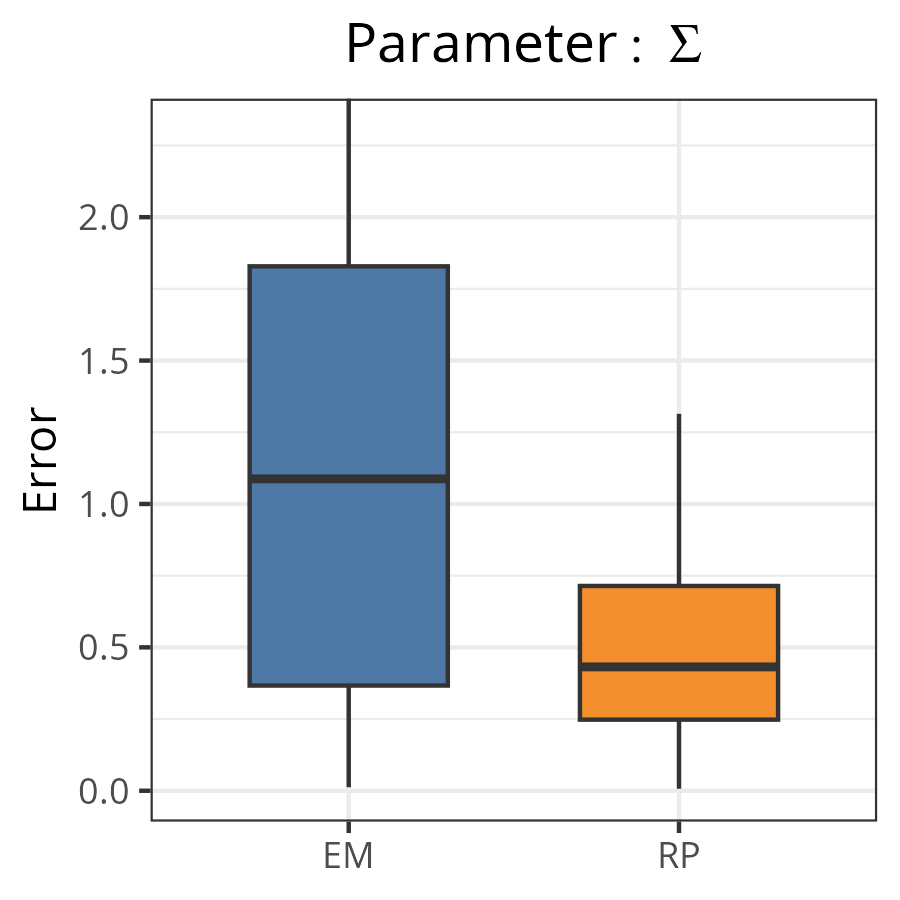}
  \caption{Boxplots of the Frobenius error $\|\Sigma-\Sigma^{*}\|_{F}$ over the $100$ Monte Carlo replicates, comparing EM and RP.}
  \label{F:sigma}
\end{minipage}

\end{figure}
\section{\small Descriptive Summary Tables and Confusion Matrices for the Simulations in Section 3}\label{A:confusion}
\label{apend2}

\begin{table}[!htbp]
\centering
\small

\begin{minipage}[t]{0.49\textwidth}
\centering
\setlength{\tabcolsep}{10pt}
\renewcommand{\arraystretch}{1.25}
\resizebox{\linewidth}{!}{%
\begin{tabular}{c c c}
\hline
$\eta$ & EM-st & RP \\
\hline
1/2 & \(\vcenter{\hbox{$\begin{pmatrix}0.163 & 0.139 \\ 0.310 & 0.389\end{pmatrix}$}}\) & \(\vcenter{\hbox{$\begin{pmatrix}0.133 & 0.169 \\ 0.253 & 0.446\end{pmatrix}$}}\) \\
1.0 & \(\vcenter{\hbox{$\begin{pmatrix}0.081 & 0.188 \\ 0.171 & 0.561\end{pmatrix}$}}\) & \(\vcenter{\hbox{$\begin{pmatrix}0.150 & 0.118 \\ 0.200 & 0.532\end{pmatrix}$}}\) \\
3/2 & \(\vcenter{\hbox{$\begin{pmatrix}0.203 & 0.093 \\ 0.294 & 0.409\end{pmatrix}$}}\) & \(\vcenter{\hbox{$\begin{pmatrix}0.198 & 0.099 \\ 0.139 & 0.564\end{pmatrix}$}}\) \\
2.0 & \(\vcenter{\hbox{$\begin{pmatrix}0.156 & 0.141 \\ 0.120 & 0.583\end{pmatrix}$}}\) & \(\vcenter{\hbox{$\begin{pmatrix}0.206 & 0.091 \\ 0.097 & 0.607\end{pmatrix}$}}\) \\
\hline
\end{tabular}%
}
\caption{Average confusion matrices (joint proportions) between true labels and posterior classifications for EM-st and RP.}
\label{Tb:confusion_all1}
\end{minipage}\hfill
\begin{minipage}[t]{0.49\textwidth}
\centering
\setlength{\tabcolsep}{10pt}
\renewcommand{\arraystretch}{1.25}
\resizebox{\linewidth}{!}{%
\begin{tabular}{c c c}
\hline
$\gamma$ & RobEM & RP \\
\hline
$\gamma=5\%$  & \(\vcenter{\hbox{$\begin{pmatrix}0.262 & 0.043 \\ 0.157 & 0.538 \end{pmatrix}$}}\) & \(\vcenter{\hbox{$\begin{pmatrix}0.232 & 0.072 \\ 0.063 & 0.632\end{pmatrix}$}}\) \\
$\gamma=10\%$ & \(\vcenter{\hbox{$\begin{pmatrix}0.254 & 0.046 \\ 0.140 & 0.560 \end{pmatrix}$}}\) & \(\vcenter{\hbox{$\begin{pmatrix}0.191 & 0.109 \\ 0.064 & 0.636\end{pmatrix}$}}\) \\
$\gamma=15\%$  & \(\vcenter{\hbox{$\begin{pmatrix}0.253 & 0.045\\ 0.178 & 0.524\end{pmatrix}$}}\) & \(\vcenter{\hbox{$\begin{pmatrix}0.292 & 0.076  \\ 0.103 & 0.529\end{pmatrix}$}}\) \\
\hline
\end{tabular}%
}
\caption{Average confusion matrices (joint proportions) between true labels and posterior classifications for RobEM-st and RP. }
\label{Tb:confusion_all2}
\end{minipage}

\end{table}

\begin{table}[!ht]
\centering
\caption{Average values of the estimated parameters over $100$ replicates of the algorithm. Marginal standard deviations (sd) are shown in parentheses. Results are reported for four separability scenarios indexed by the location-shift parameter $\eta \in \{1/2,1,3/2,2\}$.}
\label{Tb:mixla_eta}
\small
\setlength{\tabcolsep}{5pt}
\renewcommand{\arraystretch}{1.12}
\begin{tabular}{|c|c|c|c|c|c|c|}
\hline
Algorithm &Component  &$\mu_{1}$ & $\mu_{2}$ & $\sigma_{11}$ & $\sigma_{22}$ & $\sigma_{21}$ \\
\hline
\multicolumn{2}{|c}{} & \multicolumn{5}{|c|}{Average estimated values} \\
\hline

Values &1&0 & 0 & 1 & 1/2 & 0 \\
&2&$\eta$ & 0 & 1/2 & 1 & 0 \\
\hline

\multicolumn{7}{|c|}{\textbf{Scenario: $\eta = 1/2$}}\\
\hline
RP &1 &-0.36 & -0.04 & 0.92 & 0.90 & 0.00 \\
&sd &(0.94) & (1.30) & (1.35) & (0.65) & (0.62) \\
&2 &0.92 & 0.04 & 0.66 & 1.02 & -0.05 \\
&sd &(0.65) & (1.37) & (0.59) & (0.88) & (0.48) \\
\hline
EM-st &1 &0.03 & -0.05 & 0.91 & 0.63 & 0.00 \\
&sd &(1.04) & (0.98) & (2.61) & (0.85) & (0.90) \\
&2 &0.89 & -0.01 & 0.50 & 0.64 & 0.01 \\
&sd &(0.66) & (0.69) & (0.48) & (0.83) & (0.54) \\
\hline

\multicolumn{7}{|c|}{\textbf{Scenario: $\eta = 1$}}\\
\hline
RP &1 &-0.38 & 0.11 & 0.84 & 0.86 & 0.03 \\
&sd &(0.71) & (1.25) & (0.48) & (0.49) & (0.33) \\
&2 &1.35 & -0.01 & 0.57 & 0.80 & 0.02 \\
&sd &(0.49) & (1.07) & (0.25) & (0.45) & (0.22) \\
\hline
EM-st &1 &0.41 & -0.01 & 0.55 & 0.50 & -0.05 \\
&sd &(0.91) & (0.69) & (0.49) & (0.40) & (0.50) \\
&2 &1.48 & -0.01 & 0.71 & 0.79 & -0.28 \\
&sd &(0.98) & (0.76) & (2.65) & (2.61) & (2.65) \\
\hline

\multicolumn{7}{|c|}{\textbf{Scenario: $\eta = 3/2$}}\\
\hline
RP &1 &-0.14 & -0.03 & 1.51 & 1.10 & -0.36 \\
&sd &(1.00) & (1.12) & (5.75) & (3.47) & (4.25) \\
&2 &1.73 & -0.04 & 0.66 & 0.97 & 0.00 \\
&sd &(0.45) & (1.11) & (0.31) & (0.99) & (0.45) \\
\hline
EM-st &1 &0.87 & 0.13 & 1.12 & 0.75 & 0.17 \\
&sd &(0.69) & (0.97) & (3.42) & (0.90) & (1.59) \\
&2 &1.39 & 0.10 & 0.44 & 0.64 & -0.02 \\
&sd &(0.65) & (1.05) & (0.28) & (0.24) & (0.30) \\
\hline

\multicolumn{7}{|c|}{\textbf{Scenario: $\eta = 2$}}\\
\hline
RP &1 &-0.27 & 0.01 & 1.10 & 0.65 & -0.07 \\
&sd &(0.39) & (0.44) & (0.43) & (0.50) & (0.34) \\
&2 &2.20 & 0.06 & 0.67 & 0.89 & 0.00 \\
&sd &(0.27) & (0.41) & (0.28) & (0.33) & (0.17) \\
\hline
EM-st &1 &0.89 & -0.03 & 0.66 & 0.64 & -0.09 \\
&sd &(0.53) & (0.88) & (0.39) & (0.33) & (0.44) \\
&2 &1.60 & -0.07 & 0.36 & 0.74 & -0.02 \\
&sd &(0.49) & (1.02) & (0.14) & (0.26) & (0.18) \\
\hline

\end{tabular}
\end{table}


\begin{table}[!ht]
\centering
\caption{Average values of the estimated parameters over $100$ replicates of the algorithm. Marginal standard deviations (sd) are shown in parentheses. Results are reported for three contamination scenarios indexed by the outlier proportion $\gamma \in \{0.05,0.10,0.15\}$.}
\label{Tb:mixla_gamma}
\small
\setlength{\tabcolsep}{5pt}
\renewcommand{\arraystretch}{1.12}
\begin{tabular}{|c|c|c|c|c|c|c|}
\hline
Algorithm &Component  &$\mu_{1}$ & $\mu_{2}$ & $\sigma_{11}$ & $\sigma_{22}$ & $\sigma_{21}$ \\
\hline
\multicolumn{2}{|c}{} & \multicolumn{5}{|c|}{Average estimated values} \\
\hline

Values &1&0 & 0 & 1 & 1/2 & 0 \\
&2&2 & 0 & 1/2 & 1 & 0 \\
\hline

\multicolumn{7}{|c|}{\textbf{Scenario: $\gamma = 0.05$}}\\
\hline
RP &1 &-0.27 & 0.12 & 1.16 & 0.68 & 0.00 \\
&sd &(0.32) & (0.54) & (0.41) & (0.53) & (0.31) \\
&2 &2.27 & 0.04 & 0.72 & 1.05 & 0.05 \\
&sd &(0.40) & (0.43) & (0.45) & (0.31) & (0.12) \\
\hline
EM-rob &1 &0.60 & 0.10 & 3.18 & 1.32 & 0.17 \\
&sd &(0.32) & (0.17) & (0.75) & (0.39) & (0.43) \\
&2 &2.07 & 0.05 & 0.89 & 2.31 & 0.06 \\
&sd &(0.10) & (0.17) & (0.37) & (0.51) & (0.16) \\
\hline

\multicolumn{7}{|c|}{\textbf{Scenario: $\gamma = 0.10$}}\\
\hline
RP &1 &-0.14 & 0.46 & 1.08 & 0.73 & 0.07 \\
&sd &(0.45) & (0.86) & (0.27) & (0.40) & (0.20) \\
&2 &2.24 & -0.11 & 0.68 & 1.09 & 0.05 \\
&sd &(0.37) & (0.68) & (0.18) & (0.40) & (0.14) \\
\hline
EM-rob &1 &0.70 & 0.17 & 3.33 & 1.61 & 0.34 \\
&sd &(0.29) & (0.23) & (0.68) & (0.66) & (0.70) \\
&2 &2.07 & 0.13 & 0.93 & 2.53 & 0.06 \\
&sd &(0.10) & (0.23) & (0.36) & (0.74) & (0.18) \\
\hline

\multicolumn{7}{|c|}{\textbf{Scenario: $\gamma = 0.15$}}\\
\hline
RP &1 &0.12 & 0.59 & 1.03 & 0.88 & 0.09 \\
&sd &(0.63) & (1.75) & (0.35) & (0.40) & (0.24) \\
&2 &2.21 & 0.28 & 0.70 & 1.12 & 0.09 \\
&sd &(0.42) & (1.57) & (0.25) & (0.81) & (0.21) \\
\hline
EM-rob &1 &0.77 & 0.28 & 3.37 & 1.85 & 0.54 \\
&sd &(0.31) & (0.30) & (0.73) & (0.87) & (0.84) \\
&2 &2.08 & 0.21 & 1.07 & 2.80 & 0.08 \\
&sd &(0.11) & (0.29) & (0.36) & (1.02) & (0.22) \\
\hline

\end{tabular}
\end{table}

\end{appendixcompact}
\end{appendices}

\FloatBarrier

\end{document}